\documentclass[preprint,10pt]{elsarticle_2}

\biboptions{square,sort,comma,numbers}

\usepackage[latin1]{inputenc}
\usepackage{latexsym}
\usepackage{graphicx}

\usepackage{stmaryrd}

\usepackage{array}
\newcolumntype{C}[1]{>{\centering\arraybackslash}p{#1}}

\usepackage{amscd}
\usepackage{amssymb,amsmath,amsthm}
\usepackage{mathptmx}
\usepackage{mathrsfs}
\usepackage{color}
\usepackage{xspace}
\usepackage{bussproofs}

\usepackage{tikz}
\usepackage{bpextra}

\usepackage{centernot}

\usepackage{hyperref}

\usepackage{cleveref}

\usepackage{relsize}

\usepackage{caption}

\allowdisplaybreaks[4]

\DeclareSymbolFont{letters}{OML}{cmm}{m}{it}

\DeclareMathAlphabet{\mathcal}{OMS}{cmsy}{m}{n}

\usepackage{array, multirow}

\newtheorem{theorem}{Theorem}[section]

\newtheorem{lemma}[theorem]{Lemma}
\newtheorem{proposition}[theorem]{Proposition}

\newtheorem*{claim*}{Claim}

\theoremstyle{definition}
\newtheorem{definition}[theorem]{Definition}
\newtheorem{example}[theorem]{Example}

\newtheorem*{theorem*}{Theorem}

\newenvironment{proofclaim}[1][the claim]{\par
\noindent \emph{Proof of #1.} }
{\hfill$\dashv$\vspace{4pt}}

\newcommand{\PD}{\ensuremath{\mathbf{PD}}\xspace}

\newcommand{\CPC}{\ensuremath{\mathbf{CPL}}\xspace}
\newcommand{\CPL}{\ensuremath{\mathbf{CPL}}\xspace}

\newcommand{\PT}{\ensuremath{\mathbf{PT}^+}\xspace}
\newcommand{\PTw}{\ensuremath{\mathbf{PT}}\xspace}
\newcommand{\MPT}{\ensuremath{\mathbf{FPT}}\xspace}

\newcommand{\PU}{\ensuremath{\mathbf{PU}^+}\xspace}
\newcommand{\PUw}{\ensuremath{\mathbf{PU}}\xspace}

\newcommand{\PInd}{\ensuremath{\mathbf{PI}}\xspace}
\newcommand{\PI}{\ensuremath{\mathbf{PI}}\xspace}

\newcommand{\PInem}{\ensuremath{\PInd^{+}}\xspace}

\newcommand{\PDs}{\ensuremath{\mathbf{PD}^+}\xspace}

\newcommand{\PInc}{\ensuremath{\mathbf{PInc}}\xspace}
\newcommand{\PIncs}{\ensuremath{\mathbf{PInc}^+}\xspace}

\newcommand{\ECL}{\ensuremath{\mathbf{CPL^+}}\xspace}
\newcommand{\ECLp}{\ensuremath{\mathsf{L}}\xspace}

\newcommand{\dep}{\ensuremath{\mathop{=\!}}\xspace}

\newcommand{\sor}{\ensuremath{\otimes}\xspace}
\newcommand{\bor}{\ensuremath{\vee}\xspace}

\newcommand{\nsor}{\ensuremath{\varoast}\xspace}

\newcommand{\bignsor}{\mathop{\raisebox{-1.9pt}{\ensuremath{\huge\text{\nsor}}}}}

\DeclareMathOperator*{\bigsor}{\bigotimes}
\DeclareMathOperator*{\bigbor}{\bigvee}

\newcommand{\nem}{\ensuremath{\mathop{\mbox{\small\rm N\hspace{-0.2pt}E}}}\xspace}

\newcommand{\ci}{\ensuremath{\wedge\textsf{I}}\xspace}

\newcommand{\ce}{\ensuremath{\wedge\textsf{E}}\xspace}

\newcommand{\sorwe}{\ensuremath{\sor\textsf{E}^-}\xspace}
\newcommand{\bote}{\ensuremath{\mathbf{\bot}\textsf{E}}\xspace}
\newcommand{\sorws}{\ensuremath{\sor\textsf{Sub}^-}\xspace}
\newcommand{\bori}{\ensuremath{\bor\textsf{I}}\xspace}
\newcommand{\bore}{\ensuremath{\bor\textsf{E}}\xspace}

\newcommand{\exclmid}{\ensuremath{\textsf{EM}_0}\xspace}

\newcommand{\com}{\ensuremath{\textsf{Com}}\xspace}
\newcommand{\ass}{\ensuremath{\textsf{Ass}}\xspace}
\newcommand{\dstr}{\ensuremath{\textsf{Dstr}}\xspace}

\newcommand{\boti}{\ensuremath{\bot\textsf{I}}\xspace}

\newcommand{\alphai}{\ensuremath{\alpha\textsf{I}}\xspace}
\newcommand{\sealpha}{\ensuremath{\textsf{SE}_\alpha}\xspace}

\newcommand{\indi}{\ensuremath{\textsf{IndI}}\xspace}

\newcommand{\see}{\ensuremath{\textsf{SE}_2}\xspace}
\newcommand{\sen}{\ensuremath{\textsf{SE}_1}\xspace}
\newcommand{\sei}{\ensuremath{\textsf{SE}_{\textsf{Ind}}}\xspace}

\newcommand{\dstrs}{\ensuremath{\textsf{Dstr}^\ast\!}\xspace}

\newcommand{\wexfalso}{\ensuremath{\it{\textsf{ex falso}}^-}\xspace}
\newcommand{\sexfalso}{\ensuremath{\it{\textsf{ex falso}}^+}\xspace}

\newcommand{\sorwi}{\ensuremath{\sor\textsf{I}^-}\xspace}

\newcommand{\sctri}{\ensuremath{\mathbf{0}\textsf{I}}\xspace}
\newcommand{\sctre}{\ensuremath{\mathbf{0}\textsf{E}}\xspace}

\newcommand{\sctrc}{\ensuremath{\mathbf{0}\textsf{Ctr}}\xspace}

\newcommand{\sorwk}{\ensuremath{\sor\textsf{W}}\xspace}

\newcommand{\nemi}{\ensuremath{\nem\textsf{I}}\xspace}

\newcommand{\Sub}{\ensuremath{\mathsf{Sub}}\xspace}

  \usetikzlibrary{automata,positioning,arrows}
\usepackage{listings}
  
  \lstdefinelanguage{pseudo}{
    morekeywords={if,elseif,then,return,end,choose,guess,when,for,foreach,while,case},
    morekeywords=[3]{false,true,and,or,not},
    morecomment=[l]{//}
  }
\newcommand{\problemdef}[3]{%
  \begin{list}{}{%
    \parsep 0pt
    \itemsep 0ex plus 0.1ex 
    \settowidth{\labelwidth}{\it Question:}
    \leftmargin\labelwidth \advance\leftmargin by \labelsep \advance\leftmargin by \leftmargini
  }
    \item[{\it Problem:\hfill}] #1
    \item[{\it Input:\hfill}] #2
    \item[{\it Question:\hfill}] #3
  \end{list}%
}

\newtheorem{namedproofcontent}{}

\definecolor{darkgreen}{rgb}{0,.7,0}

\usepackage[normalem]{ulem}

\newcommand{\logicFont}[1]{\mathcal{#1}}
\newcommand{\classFont}[1]{\mathsf{#1}}

\newcommand{\halfliteral}[1]{\protect\ensuremath{#1}}
\newcommand{\literal}[1]{\halfliteral{#1}\xspace}

\newcommand{\set}[3][]{\literal{\left\{#2\;\middle|\;\ifthenelse{\equal{#1}{}}{\text{#3}}{\parbox{#1}{#3}}\right\}}}

\newcommand{\logic}[1]{\literal{\logicFont{#1}}}
\newcommand{\paraLogic}[2]{\ensuremath{\logic{#1}\ifthenelse{\equal{#2}{}}{}{(#2)}}\xspace}

\newcommand{\LL}{\ensuremath{\mathsf{L}}\xspace}

\newcommand{\class}[1]{\literal{\classFont{#1}}}

\newcommand{\TIME}{\class{TIME}}
\newcommand{\NTIME}{\class{NTIME}}
\newcommand{\ATIME}{\class{ATIME}}
\newcommand{\SPACE}{\class{SPACE}}
\newcommand{\NSPACE}{\class{NSPACE}}
\newcommand{\ASPACE}{\class{ASPACE}}

  \xspaceaddexceptions{\TIME \NTIME \ATIME \SPACE \NSPACE \ASPACE}

\journal{arXiv.org}

\begin{document}

\begin{frontmatter}



\title{Propositional Team Logics\tnoteref{tn}}

\tnotetext[tn]{Some results in this paper were included in the dissertation of the first author \cite{Yang_dissertation}, which was supervised by the second author.}

 \author[FY]{Fan Yang\corref{cor}\fnref{fnf}}

  \address[FY]{Department of Values, Technology and Innovation, Delft University of Technology, Jaffalaan 5, 2628 BX Delft, The Netherlands}
\ead{fan.yang.c@gmail.com}

\fntext[fn]{This paper is partially based on research that was carried out in the Graduate School of Mathematics and Statistics of the University of Helsinki. 
}

 \author[JV]{Jouko V\"{a}\"{a}n\"{a}nen\fnref{fnj}}

  \ead{jouko.vaananen@helsinki.fi} 
 \address[JV]{Department of Mathematics and Statistics, Gustaf H\"{a}llstr\"{o}min katu 2b, PL 68, FIN-00014 University of Helsinki, Finland and University of Amsterdam, The Netherlands}

\fntext[fnj]{The research was partially supported by grant 251557 of the Academy of Finland.}

\cortext[cor]{Corresponding author}



\begin{abstract}
We consider  team semantics for propositional logic, continuing \cite{VY_PD}. In team semantics the truth of a propositional formula is considered in a {\em set} of valuations, called a {\em team}, rather than in an individual valuation. This offers the possibility to give  meaning to concepts such as  dependence, independence and inclusion.
We associate with every formula $\phi$ based on finitely many propositional variables  the set $\llbracket\phi\rrbracket$ of {\em teams} that satisfy  $\phi$. We define a full propositional team logic in which every  set of teams is definable as  $\llbracket\phi\rrbracket$ for suitable $\phi$. This requires going beyond the logical operations of classical propositional logic. 
%
We exhibit  a hierarchy of logics between the smallest, viz. classical propositional logic, and the full propositional team logic. We characterize these different logics in several ways: first syntactically by their logical operations, and then semantically by the kind of sets of teams they are capable of defining. In several important cases we are able to find complete axiomatizations for these logics.

\end{abstract}


\end{frontmatter}

\section{Introduction}

In classical propositional logic the propositional atoms, say $p_1,\ldots, p_n$, 
are given a truth value $1$ or $0$ by what is called a valuation and then any propositional formula $\phi$ can be associated with the set $|\phi|$ of valuations giving $\phi$ the value $1$. This constitutes   a perfect analysis of the circumstances under which $\phi$ is true. The formula $\phi$ can be presented in  so-called Disjunctive Normal Form based on taking the disjunction of descriptions of the valuations in $|\phi|$. Two fundamental results can be proved for classical propositional logic. The first says that   {\em every} set of valuations of $p_1,\ldots, p_n$ is equal to $|\phi|$ for some propositional formula $\phi$. The second fundamental result says that there is a simple {\em complete} axiomatization of those $\phi$ that are valid in the sense that $|\phi|$ is the full set of all valuations on the propositional atoms occurring in $\phi$.

In this paper, which continues \cite{VY_PD}, we consider a richer semantics called {\em team semantics} for propositional logic. In team semantics the truth of a propositional formula is evaluated in a {\em set} of valuations, called a {\em team}, rather than in an individual valuation. This offers the possibility of considering {\em probabilities} of formulas, as in \cite{MR3423958}, and the meaning of concepts such as {\em dependence, independence} and {\em inclusion}, as in \cite{VY_PD}. It is the latter possibility that is our focus in this paper.  

Team semantics was introduced by the second author in \cite{Van07dl} on the basis of  a new compositional semantics, due to Hodges \cite{Hodges1997a,Hodges1997b}, for independence friendly logic \cite{HintikkaSandu1989,IF_book}. The monograph \cite{Van07dl} was written in the context of predicate logic and team semantics was used to give meaning to a variable being totally determined by a sequence of other variables. In the context of propositional and modal logic team semantics was introduced  in \cite{VaMDL08}. In propositional logic team semantics can be used to give meaning to a propositional variable being totally determined by a sequence of other variables. It took a few years before this idea was fully exploited in  \cite{Virtema2014,Yang_dissertation}.  Meanwhile 
modal dependence logic, i.e. team semantics for modal logic, was investigated e.g. in \cite{sevenster09,eblo11,EHMMVV2013,lovo10,MID_mc,HLSV14,HellaStumpf2015}.

When propositional formulas are evaluated in a team---i.e. a set---of valuations, a whole new landscape opens in front of us. The first observation is a numerical explosion: If we have $n$ propositional atoms, there are $2^n$ valuations, $2^{2^n}$ teams, and $2^{2^{2^n}}$ sets of teams. For $n=3$ the third number is about $10^{77}$. This emphasises the need for mathematical methods in team semantics.  The truth table methods which list all possibilities is bad enough in ordinary propositional logic, but totally untenable in team semantics.

In classical propositional logic, we associate with every formula $\phi$ based on propositional atoms $p_1,\ldots,p_n$  the set $|\phi|$ of {\em valuations} that satisfy $\phi$. Similarly, in team semantics we associate with every formula $\phi$ based on propositional atoms $p_1,\ldots,p_n$  the set $\llbracket\phi\rrbracket$ of {\em teams} that satisfy (in the sense defined below)  $\phi$. By choosing our formulas carefully we can express {\em every} set of teams in the form $\llbracket\phi\rrbracket$ for suitable $\phi$, but this requires going beyond the logical operations of classical propositional logic. We can also axiomatize the propositional formulas that are {\em valid} i.e. satisfied by every team.

The rich structure of teams gives rise to a plethora of new propositional connectives. Most importantly, disjunction has several versions. To define when a team $X$  satisfies $\phi\vee\psi$ we can say that this happens if $X$ satisfies $\phi$ or it satisfies $\psi$, or we can say that this happens if $X$ is the union of two sets $Y$ and $Z$ such that $Y$ satisfies $\phi$ and $Z$ satisfies $\psi$, or, finally, we can also say that this happens if, assuming  $X\ne\emptyset$, the team $X$ is the union of two sets $Y\ne\emptyset$ and $Z\ne\emptyset$ such that $Y$ satisfies $\phi$ and $Z$ satisfies $\psi$. If $X$ is a singleton, which corresponds to the classical case, the first two disjunctions are equivalent, but the third is equivalent to $\phi\wedge\psi$.  But for non-singleton teams there is a big difference in every respect. These distinctions, leading to different variants of familiar logical operations, reveal  a hierarchy of logics between the smallest, viz. classical propositional logic, and the maximal one capable of defining every set of teams. We characterize these different logics in several ways: first syntactically by their logical operations, and then semantically by the kind of sets of teams they are capable of defining. In several important cases we are able to find complete axiomatizations for these logic.

In our previous paper \cite{VY_PD} we considered sets of teams that are {\em downward closed} in the sense that if a team is in the set, then every subteam is in the set, too. Respectively, the logics studied in \cite{VY_PD} have the property that the sets of teams defined by their formulas are downward closed. We isolated five  equivalent logics with this property, all based on some aspect of {\em dependence}. In these logics every downward closed set of teams is definable, and the logics have complete axiomatizations. The axiomatizations are by no means as simple as typical axiomatizations of classical propositional logic, but have still a certain degree of naturality.

In this paper we consider sets of teams, and related propositional logics, that are not downward closed. A property in a sense opposite to downward closure is closure under (set-theoretical) unions. In fact, a set of teams that is both closed downward and closed under unions is definable in classical propositional logic. So-called {\em inclusion logic}, to be defined below, is an example of a logic in which definable sets of teams are closed under unions. So-called {\em independence logic}, also to be defined below, is neither downward closed nor closed under unions. Our methods do not seem to apply to independence logic, we can merely approximate it from below and from above with logics that we understand better.

We do not rule out the possibility that a team is empty. Accordingly we distinguish whether a set of teams contains the empty team as an element or not. The basic dependence, independence and inclusion logics have the Empty Team Property i.e. every definable set of teams contains the empty team. However, many of our proofs depend on the ability to express the non-emptiness of a team. For this purpose we also consider a special atomic formula $\nem$ the only role of which is to say that the team is nonempty.
This so-called non-emptiness $\nem$ was introduced in \cite{Yang_dissertation} and in \cite{Vaananen_multiverse}. We give examples which suggest that $\nem$ is not completely alien to common usage of language although it seems hopelessly abstract.  The introduction of $\nem$ leads to two versions of each of our propositional logics: one without $\nem$ and one with $\nem$. 

This paper is structured as follows. In Section~\ref{premin} we define the basic concepts and make some preliminary observations. 
We also define the propositional team logics we study in the paper, including \emph{propositional dependence logic}, \emph{propositional independence logic}, \emph{propositional union closed logic}, \emph{propositional inclusion logic} and \emph{propositional team logic} as well as the strong version of each. 
In Section~\ref{pisec:nf} we establish basic normal forms and use them to obtain semantic characterizations of our logics, whether strong or not. In Section \ref{sec:meta_prop} we prove some metalogical properties of our logics, including compactness and the closure under classical substitutions of the logics. In Section~\ref{axioma} we establish complete axiomatizations of the strong versions of our logics. Several open problems are listed in the concluding Section~\ref{conc}.

\section{Preliminaries}\label{premin}




Our propositional team logic follows the pattern set forth on first-order level by {\em dependence logic} \cite{Van07dl}, {\em independence logic} 
\cite{D_Ind_GV}, as well as  inclusion and exclusion logics \cite{Pietro_I/E}. 
The concepts of dependence and independence were earlier introduced in database theory, starting with  \cite{DBLP:persons/Codd71a}. 
However, in database theory the focus is on dependence and independence of attributes per se, while we take the dependence and independence as atomic formulas and use logical operations to build complex formulas. 
The benefit of considering complex formulas is that we can express very involved types of dependence and independence. A good example is the fact that first-order inclusion logic can express in finite models exactly all dependencies expressible in fixed point logic \cite{DBLP:conf/csl/GallianiH13}.


We follow here the reasoning of Wilfrid Hodges \cite{Hodges1997a,Hodges1997b} to the effect that a set of valuations, rather than a single valuation,  permits the delineation of dependence and independence. We call such sets {\em teams}. Let us now give the formal definition of a team.







\begin{definition}  
Throughout the paper we fix an infinite set ${\rm Prop}=\{p_i\mid i\in \mathbb{N}\}$ of propositional variables. We sometimes use $\vec{x}, \vec{y},\vec{z},\dots$ to denote arbitrary sequences of propositional variables. 
A \emph{valuation} $s$ on a set $N$ of indices (i.e. a set of natural numbers) is a function from $N$ to the set $2=\{0,1\}$. A \emph{team} $X$ on $N$ is a set of valuations on $N$. A \emph{team} $X$ on the set $\mathbb{N}$ of all natural numbers is called a \emph{team}.
If $X$ is a team on $N$ and $N'\subseteq N$, then we write $X\upharpoonright N'$ for the set $\{s\upharpoonright N'\mid s\in X\}$.
\end{definition}

\Cref{tab:team} shows an example of a team $X$ consisting of six valuations. One possibility is to view a team as an information state as is done in {\em inquisitive logic} \cite{InquiLog}. The idea is that there is one ``true" valuation $v$ and the valuations in the team are approximations of it as far as we know. The bigger the team the bigger is our uncertainty about $v$. On the other hand, if the team is as small as a singleton $\{v\}$, we know the valuation, and there is no uncertainty.
This is just one intuition behind the team concept. A different intuition is that the valuations in a team arise from scientific observations. They may arise also from the organizational structure of a large company, etc.

We call propositional logics that have  semantics based on teams \emph{propositional team logics}. 
%
As the first step, let us examine the usual classical propositional logic in the setting of team semantics.

\begin{definition}\label{syntax_cpc}
Well-formed
formulas of \emph{classical propositional  logic} (\CPC)  are given by the following grammar 
\[
    \phi::= \,p_i\mid \neg p_i\mid\bot\mid(\phi\wedge\phi)\mid(\phi\sor\phi).
\] 
Here we use the symbol $\sor$ to denote the disjunction of \CPL. A well-formed formula of \CPC is said to be a \emph{formula in the language of \CPC} or a \emph{classical formula}.
\end{definition}

 \begin{table}[t]
\begin{center}
\begin{tabular}{c|ccccc}
&$p_0$&$p_1$&$p_2$&$p_3$&$\dots$\\
\hline
$s_1$&1&1&1&1\\
$s_2$&1&0&0&0\\
$s_3$&0&1&1&1\\
$s_4$&0&0&0&0&$\dots$\\
$s_5$&1&1&0&0\\
$s_6$&0&1&0&1\\
\hline
\end{tabular}
\caption{A team $X=\{s_1,\dots,s_6\}$\label{tab:team}}
\end{center}
\end{table}

\begin{definition}\label{TS_CT}
We  define inductively the notion of a classical formula $\phi$  being \emph{true} on a team $X$, denoted by $X\models\phi$, as follows:
\begin{itemize}
\item $X\models p_i$ iff
for all $s\in X$, $s(i)=1$
  \item $X\models\neg p_i$  iff
for all $s\in X$, $s(i)=0$
  \item $X\models\bot$ iff $X=\emptyset$
  \item $X\models\phi\wedge\psi$ iff $X\models\phi$ and
  $X\models\psi$
  \item $X\models\phi\sor\psi$ iff there exist two subteams $Y,Z\subseteq X$ with $X=Y\cup Z$ such that
  \(Y\models\phi\text{ and }Z\models\psi\)
\end{itemize}
\end{definition}

We write $\phi(p_{i_1},\dots,p_{i_n})$ if the propositional variables occurring in the formula $\phi$ are among $p_{i_1},\dots,p_{i_n}$. The following lemma summarizes the main properties of classical formulas. The reader is referred to \cite{VY_PD} for details on other properties of the team semantics of classical formulas. 

\begin{lemma}\label{cpl_prop}
Classical formulas have the \emph{Locality Property}, the \emph{Flatness Property}, the \emph{Downward Closure Property}, the \emph{Union Closure Property},  and the \emph{Empty Team Property} defined as follows.
\begin{description}
\item[(Locality Property)] Let $X$ and $Y$ be two teams, and $\phi(p_{i_1},\dots,p_{i_n})$ a formula. 
If $X\upharpoonright \{{i_1},\dots,{i_n}\}=Y\upharpoonright \{{i_1},\dots,{i_n}\}$, then
\(X\models\phi\iff Y\models\phi.\) 
\item[(Flateness Property)] \(X\models\phi\iff\forall s\in X(\{s\}\models\phi)\)
\item[(Downward Closure Property)] If $X\models\phi$ and $Y\subseteq X$, then $Y\models\phi$
\item[(Union Closure Property)] If $X\models\phi$ for all $X\in\mathcal{X}$, then $\bigcup\mathcal{X}\models\phi$
\item[(Empty Team Property)] $\emptyset\models\phi$ always holds 

\end{description}

\end{lemma}

Under the usual single valuation semantics a classical formula $\phi(p_{i_1},\dots,p_{i_n})$ defines a set $|\phi|=\{s\in 2^N:s\models\phi\}$ of valuations (a team!) on $N=\{{i_1},\dots,{i_n}\}$; the same formula under the team semantics defines a set
\[\llbracket \phi\rrbracket:=\{X\subseteq 2^N\mid X\models\phi\}\]
of teams on $N$. It is well-know that \CPL is \emph{expressively complete} under the usual single valuation semantics in the sense that every property $X\subseteq 2^N$ is definable by a classical formula $\phi$, i.e., $X=|\phi|$. We now define a similar notion of \emph{expressive completeness} for a set of team properties under the team semantics. 

\begin{definition}\label{expressive_comp_df}
Let $\mathbb{P}$ be a set of team properties i.e. a set of sets of teams. We let
$\mathbb{P}_N=\{\mathsf{P}\upharpoonright N: \mathsf{P}\in\mathbb{P}\}$, where each $\mathsf{P}\upharpoonright N=\{X\upharpoonright N:X\in\mathsf{P}\}$ is a team property on a finite set $N$ of indices.
  We say that a propositional team logic \LL \emph{characterizes}  $\mathbb{P}$, if for each index set $N=\{i_1,\dots,i_n\}$, 
\[\mathbb{P}_N=\{\llbracket \phi\rrbracket:~\phi(p_{i_1},\dots,p_{i_n})\text{ is a formula in the language of \LL}\}.\] If a logic characterizes a set $\mathbb{P}$ of team properties, then we also say that the logic is \emph{expressively complete} for $\mathbb{P}$.
\end{definition}

Below we define some interesting team properties, already inherent in Lemma~\ref{cpl_prop}:

\begin{definition}\label{team_prop_df}
A team property $\mathsf{P}$, i.e., a set of teams, is called 
\begin{itemize}
\item  \emph{flat} if  $X\in \mathsf{P}\iff \forall s\in X(\{s\}\in\mathsf{P})$;
\item \emph{downward closed} if  $Y\subseteq X\in \mathsf{P}\Longrightarrow Y\in\mathsf{P}$;
\item \emph{union closed} if $\mathcal{X}\subseteq \mathsf{P}\Longrightarrow \bigcup\mathcal{X}\in\mathsf{P}$.
\end{itemize}

\end{definition}

It follows from our previous paper \cite{VY_PD} that several propositional logics of dependence (\PD) (including propositional dependence logic and inquisitive logic) are expressively complete for the set of all nonempty downward closed team properties. 
In this paper we will study logics that are expressively complete for each of the team properties defined above. In particular, we will prove that \CPL is expressively complete for the set of flat team properties  and it is the biggest propositional team logic that defines both all downward closed team properties and all union closed team properties (\Cref{cpl_exp_compl}).



The empty team is  a member of any flat team property and of any nonempty downward closed team property. The familiar classical formulas and formulas in the language of \PD that we studied in our previous paper \cite{VY_PD}  all have the empty team property. To define team properties that do not contain the empty team, we introduce a new atom \nem, called \emph{non-emptiness}, stating that the team in question is nonempty.  To define also other interesting team properties, we now enrich the language of our logic.


\begin{definition}\label{syntax_pt}
Well-formed formulas of the \emph{full propositional  team logic} (\MPT)  are given by the following grammar
\[\begin{split}
    \phi::=& \,p_i\mid \neg p_i\mid \nem\mid\bot\mid p_{i_1}\dots p_{i_k} \perp p_{j_1}\dots p_{j_m}\mid \dep(p_{i_1},\dots, p_{i_k},p_j)\\
    &\mid p_{i_1},\dots, p_{i_k}\subseteq p_{j_1}\dots p_{j_m} \mid(\phi\wedge\phi)\mid(\phi\sor\phi)\mid(\phi\nsor\phi)\mid (\phi\vee\phi)   
\end{split}
\] 
The formulas $p_{i_1}\dots p_{i_k} \perp p_{j_1}\dots p_{j_m}$, $\dep(p_{i_1},\dots, p_{i_k},p_j)$ and $p_{i_1},\dots, p_{i_k}\subseteq p_{j_1}\dots p_{j_m}$ are called the \emph{independence atom}, the \emph{dependence atom} and the \emph{inclusion atom}, respectively. 
The connectives $\sor$, $\nsor$ and $\vee$ are called the  \emph{tensor (disjunction)}, the \emph{nonempty disjunction} and the \emph{Boolean disjunction}, respectively. 
\end{definition}

\begin{definition}\label{TS_PT}
We  define inductively the notion of a formula $\phi$ in the language of \MPT  being \emph{true} on a team $X$, denoted by $X\models\phi$. All the cases are identical to those defined in \Cref{TS_CT} and additionally:
\begin{itemize}
\item $X\models\nem$ iff $X\neq\emptyset$
\item $X\models\, p_{i_1}\dots p_{i_k} \perp p_{j_1}\dots p_{j_m}$ iff for all $s,s'\in X$, there exists $s''\in X$ such that
\[\langle s''(i_1),\dots, s''(i_k)\rangle=\langle s(i_1),\dots, s(i_k)\rangle\]
and
\[\langle s''(j_1),\dots, s''(j_m) \rangle=\langle s'(j_1),\dots, s'(j_m)\rangle\]
\item $X\models\dep( p_{i_1}\dots p_{i_k},p_j)$ iff for all $s,s'\in X$, 
\[\text{if }\langle s(i_1),\dots, s(i_k)\rangle=\langle s'(i_1),\dots, s'(i_k)\rangle,\text{ then }s(j)=s'(j)\]
\item $X\models p_{i_1}\dots p_{i_k}\subseteq p_{j_1}\dots p_{j_k}$ iff for all $s\in X$,  there exists $s'\in X$ such that
\[\langle s(i_1),\dots, s(i_k)\rangle=\langle s'(j_1),\dots, s'(j_k)\]
\item $X\models\phi\varoast\psi$ iff $X=\emptyset$ or there are nonempty $Y$ and $Z$ such that $X=Y\cup Z$, $Y\models\phi$ and $Z\models\psi$
  \item $X\models \phi\bor\psi$ iff $X\models \phi$ or $X\models\psi$
\end{itemize}
We say that a formula $\phi$  is \emph{valid}, denoted by  $\models\phi$, if $X\models\phi$ holds for all teams $X$. 
We say that a formula $\psi$ is a \emph{logical consequence} of a set $\Gamma$ of formulas, written $\Gamma\models\psi$, if for any team $X$ such that $X\models\phi$ for all $\phi\in \Gamma$, we have $X\models\psi$. We also write $\phi\models\psi$ for $\{\phi\}\models\psi$. If $\phi\models\psi$ and $\psi\models\phi$, then we say that $\phi$ and $\psi$ are \emph{semantically equivalent}, in symbols $\phi\equiv\psi$.  

Let $\LL_1$ and $\LL_2$ be two propositional team logics. 
We write $\LL_1\leq \LL_2$ if every formula of $\LL_1$ is semantically equivalent to a formula of $\LL_2$. If $\LL_1\leq \LL_2$ and $\LL_2\leq\LL_1$, then we write $\LL_1\equiv\LL_2$ and say that $\LL_1$ and $\LL_2$ have the \emph{same expressive power}.
\end{definition}


Let us now spend a few moments with the atoms and connectives of \MPT. 

\subsection*{\underline{Independence atom}}

Let us first take a closer look at the independence atoms by considering the team $X$ of \Cref{tab:team}. It can be verified that the independence atom $p_0\perp p_3$ is satisfied by $X$. One may think of the team $X$ as given data about $p_0,p_1,p_2,p_3,\dots$. For example, $p_0,p_1,p_2$ may be propositional variables which tell whether some valves $V_0,V_1,V_2$ controlling gas flow in an industrial process are open (1) or closed (0), and $p_3$ is a propositional variable indicating whether a warning lamp is on (1) or off (0). We can conclude on the basis of the team of \Cref{tab:team} that the lamp is independent of the valve $V_0$. However, the lamp is not completely independent of $V_1$, because if $V_1$ is closed, the lamp is definitely off. Also, the lamp is not entirely independent of the valve $V_2$, because   if $V_2$ is closed, the lamp is again definitely off. 

One way to describe the truth definition of $X\models p_{i_1}\dots p_{i_k} \perp p_{j_1}\dots p_{j_m}$ is to compare it to Cartesian product: 
$X\models \{ p_i : i\in I\}\perp\{ p_j : j\in J\}$ if and only if $$X\upharpoonright I\cup J=(X\upharpoonright I)\times (X\upharpoonright J).$$
This manifests the similarity between our concept of independence and the concept of independence of random variables in statistics.


The implication problem of independence atoms (i.e., the problem of asking whether an independence atom  follows from a set of independence atoms) can be completely axiomatized  by the axioms below, known in database theory as the Geiger-Paz-Pearl axioms (\cite{Geiger-Paz-Pearl_1991}):
\begin{description}
\item[(i)] If $\vec{x}\perp \vec{y}$, then $\vec{y}\perp \vec{x}$.
\item[(ii)] If $\vec{x} \perp \vec{y}$, then $\vec{z}\perp \vec{y}$, where $\vec{z}$ is a subsequence of $\vec{x}$.
\item[(iii)] If $\vec{x}\perp\vec{y}$, then $\vec{u}\perp \vec{v}$, where $\vec{u}$ and  $\vec{v}$ are permutations of $\vec{x}$ and $\vec{y}$, respectively.
\item[(iv)] If $\vec{x}\perp\vec{y}$ and $\vec{x} \vec{y}\perp \vec{z}$, then $\vec{x} \perp  \vec{y} \vec{z}$.
\end{description}


While the downward closure property has a profound influence on  properties of dependence logic as already mentioned, the independence atoms violate this property. For example, in \Cref{tab:team} in the team $X$ the attributes $p_0$ and $p_3$ are independent but in the  subteam $Y=\{s_1,s_2,s_3\}$ they are not. We will see in the sequel that propositional  independence logic and other propositional team logics have a completely different flavor than propositional logics of dependence. 


\subsection*{\underline{Non-emptiness atom}}

Another formula that violates the downward closure property is the very simple atom $\nem$ that we call nonemptieness  which states that the team is nonempty.
An easy inductive proof shows that the \nem-free fragment of \MPT has the empty team property. But often when describing properties of teams, we do want to distinguish between the empty team and the nonempty teams. The atom \nem is introduced exactly for this purpose.

The symbol $\nem$ is a logical symbol, on a par with $\bot$, with no internal structure and no proposition symbols occurring in it. While $\bot$ is generally conceived of as a symbol of contradiction, one may ask what is the intuitive meaning of $\nem$? Does this symbol occur in natural language or in scientific discourse? Let us think of a natural language sentence that has the modality ``might": 
\begin{center}
I \emph{might} come to the party.
\end{center}
Given a nonempty information state (i.e. a team) $X$, this sentence can be characterized as ``there exists a nonempty substate $Y$ in which I indeed come to the party". This ``might'' modality (denoted by $\triangledown$) was  considered by Hella and Stumpf in \cite{HellaStumpf2015} and its team semantics is given by the clause
\begin{itemize}
\item $X\models\triangledown\phi$ iff $X=\emptyset$ or there exists a nonempty team $Y\subseteq X$ such that $Y\models\phi$
\end{itemize}
The ``might'' modality can be expressed in terms of the more basic notion of non-emptiness of a team: 
\[\triangledown\phi\equiv\bot\vee\big((\phi\wedge \nem)\otimes \top\big).\]

\subsection*{\underline{Contradictions and  linear implication}}

In the presence of the non-emptiness \nem \emph{contradiction} has  two variants: the \emph{weak contradiction} $\bot$ that is satisfied only by the empty team and the \emph{strong contradiction} $\bot\wedge\nem$ that is satisfied by no team at all. 

A related logical constant is the \emph{linear implication} $\multimap$ (introduced by Abramsky and V\"{a}\"{a}n\"{a}nen \cite{AbVan09}) that has the semantics 
\begin{itemize}
\item $X\models\phi\multimap\psi$ iff for any team $Y$, if $Y\models\phi$, then $X\cup Y\models\psi$
\end{itemize}
The strong contradiction is easily definable using the linear implication:
\[\bot\wedge\nem\equiv \top\multimap \bot,\]
where  $\top=p_{i_1}\sor\neg p_{i_1}$. The reader is referred to \cite{AbVan09} for details on linear implication. We only remark that in the presence of the downward closure property, we have 
\[\phi\models\psi\iff \emptyset \models\phi\multimap\psi.\] 
In other words, deciding whether $\psi$ is a logical consequence of $\phi$ is reduced to deciding whether the linear implication $\phi\multimap\psi$ is satisfied by the empty team.

 



\subsection*{\underline{Disjunctions}}

Due to the way we define semantics there are more propositional operations than in the case of  classical propositional logic. In particular, disjunction has three different incarnations, namely $\sor$, $\nsor$ and $\vee$. These different forms arise from the difference between considering individual valuations and sets of valuations.

The tensor disjunction $\sor$ generalizes the disjunction of classical propositional logic. The semantics of $\sor$ and other connectives as defined in \Cref{TS_CT,TS_PT} is known in the literature (see e.g.,\cite{Pietro_I/E}) as the \emph{Lax Semantics} (in contrast to the \emph{Strict Semantics}). An easy inductive proof shows that our logic \MPT has the locality property. By contrast, if we replace the clause for tensor disjunction $\sor$ in \Cref{TS_CT} by the corresponding clause under strict semantics (denoted by $\models^s$)
  \begin{itemize}
\item $X\models^s\phi\sor\psi$ iff there exist two \emph{disjoint} subteams $Y,Z\subseteq X$ with $X=Y\cup Z$ such that
  \(Y\models^s\phi\text{ and }Z\models^s\psi\)
\end{itemize}
the logic does not any more satisfy the local property. 
This is because, for instance, for the two valuations $s_0$ and $s_1$ defined in \Cref{tab:team}, we have $\{s_0,s_1\}\models^s(\nem\wedge p_0)\sor(\nem\wedge p_0)$ while $\{s_0\}\not\models^s(\nem\wedge p_0)\sor(\nem\wedge p_0)$, even though $\{s_0\upharpoonright\{0\},s_1\upharpoonright\{0\}\}=\{(0,1)\}=\{s_0\upharpoonright\{0\}\}$. We refer the reader to \cite{Pietro_I/E} for further discussions on the difference between lax and strict semantics.

The nonempty disjunction $\phi\nsor\psi$ was introduced by Raine R\"{o}nnholm (personal communication). 
One can easily verify that $\nsor$ can be defined in terms of the other disjunctions and the non-emptiness:
\[\phi\nsor\psi\equiv \bot\vee\big((\phi\wedge\nem)\sor(\psi\wedge\nem)\big).\]
Moreover, it was observed in \cite{HellaStumpf2015} that in any fragment of \MPT that has the empty team property the might modality $\triangledown$ and the nonempty disjunction $\nsor$ as inter-definable:
\begin{equation}\label{nsor_modality_interdf}
\triangledown\phi\equiv \phi\nsor\top\quad\text{ and }\quad\phi\nsor\psi\equiv(\phi\sor\psi)\wedge (\triangledown\phi\sor\triangledown\psi)
\end{equation}

To understand the meaning of $\nsor$ in natural language, let us think of the sentence (in the context of chess):
\begin{equation}\label{RQ}
 \mbox{Rook or queen was sacrificed in each play.}
\end{equation}
It is clear what it means to say that a nonempty set $X$ of plays satisfies this, and in each play in $X$ either rook or queen  was sacrificed. There is a slight difference in saying
\begin{equation}\label{RQ1}
 \mbox{Rook or queen was sacrificed and both cases occurred in some plays.}
\end{equation}
In our symbolic language, denoting ``Rook was sacrificed" by $\phi$ and 
``queen was sacrificed" by $\psi$, (\ref{RQ}) would be written
$$X\models\mbox{$\phi$ $\otimes$ $\psi$}$$
while (\ref{RQ1}) would be written
$$X\models\mbox{($\phi\wedge\nem$) $\otimes$ ($\psi\wedge\nem$)}\text{ or }X\models\phi\nsor\psi.$$
In a sense, $\phi\varoast\psi$ is an ``honest" disjunction: if the team has anything in it at all, then it is divided between $\phi$ and $\psi$ in the non-trivial way that both get a nonempty subteam. We can think that whoever says (\ref{RQ1}), means that if some plays were actually played, then in some of them a Rook was sacrificed and in some the Queen. To put it in a more general context, the formula $\phi\varoast\psi$ permits a type of ``free choice" by having each disjunct nonvoid. This way the nonempty disjunction $\phi\nsor\psi$ provides more information than the tensor disjunction $\phi\sor\psi$. In particular, uttering a disjunction with a void disjunct is actually less informative than simply stating one of the disjuncts.  

The Boolean disjunction $\vee$ was called \emph{intuitionistic disjunction} in our previous paper \cite{VY_PD} in the context of propositional logics of dependence. In particular, in the presence of the downward closure property the intuitionistic disjunction has the 
\emph{disjunction property}:  
\[\models\phi\vee\psi\text{ implies }\models\phi\text{ or }\models\psi.\]
However, in the absence of the downward closure, 
this property, reminiscent of constructive logic, for $\vee$ fails, since, e.g., $\models\bot\vee\nem$ and $\models (p\nsor \neg p)\vee (p\vee \neg p)$, whereas $\not\models\bot$, $\not\models\nem$, $\not\models p\nsor \neg p$ and $\not\models p\vee\neg p$.

We define the empty disjunction for all three disjunctions as
\[\bigsor\emptyset:=\bot,~\bignsor\emptyset:=\bot\text{ and }\bigbor\emptyset:=\bot\wedge\nem.\]

\vspace{\baselineskip}

\begin{table}[t]
\caption{Propositional team logics}
{\setlength{\tabcolsep}{0.3em}
\begin{center}
\begin{tabular}{|l|c|c|}
\multicolumn{1}{c}{\textbf{Logic}}&\multicolumn{1}{c}{\textbf{Atoms}}&\multicolumn{1}{c}{\textbf{Connectives}}\\\hline\hline
Classical propositional logic (\CPL)&$p_i,\neg p_i,\bot$&$\wedge,\sor$\\\hline
Strong classical propositional logic (\ECL)&$p_i,\neg p_i,\bot,\nem$&$\wedge,\sor$\\\hline
\multirow{2}{*}{Propositional independence logic (\PInd)}&$p_i,\neg p_i,\bot,$&\multirow{2}{*}{$\wedge,\sor$}\\
&$p_{i_1}\dots p_{i_k} \perp p_{j_1}\dots p_{j_m}$&\\\hline
\multirow{2}{*}{Strong propositional independence logic (\PInem)} &$p_i,\neg p_i,\bot,\nem,$&\multirow{2}{*}{$\wedge,\sor$}\\
&$p_{i_1}\dots p_{i_k} \perp p_{j_1}\dots p_{j_m}$&\\\hline
Propositional team logic (\PTw) &$p_i,\neg p_i,\bot$&$\wedge,\nsor,\vee$\\\hline
Strong propositional team logic (\PT) &$p_i,\neg p_i,\bot,\nem$&$\wedge,\sor,\vee$\\\hline
Propositional union closed logic (\PUw) &$p_i,\neg p_i,\bot$&$\wedge,\sor,\nsor$\\\hline
Strong propositional union closed logic (\PU) &$p_i,\neg p_i,\bot,\nem$&$\wedge,\sor,\nsor$\\\hline
\multirow{2}{*}{Propositional inclusion logic (\PInc)} &$p_i,\neg p_i,\bot,$&\multirow{2}{*}{$\wedge,\sor$}\\
&$p_{i_1}\dots p_{i_k}\subseteq p_{j_1}\dots p_{j_k}$&\\\hline
\multirow{2}{*}{Strong propositional inclusion logic (\PIncs)} &$p_i,\neg p_i,\bot,\nem,$&\multirow{2}{*}{$\wedge,\sor$}\\
&$p_{i_1}\dots p_{i_k}\subseteq p_{j_1}\dots p_{j_k}$&\\\hline
\multirow{2}{*}{Propositional dependence logic (\PD)} &$p_i,\neg p_i,\bot,$&\multirow{2}{*}{$\wedge,\sor,\vee$}\\
&$\dep(p_{i_1}\dots p_{i_k},p_j)$&\\\hline
\multirow{2}{*}{Strong propositional dependence logic ($\PD^+$)} &$p_i,\neg p_i,\bot,$&\multirow{2}{*}{$\wedge,\sor,\vee,\multimap$}\\
&$\dep(p_{i_1}\dots p_{i_k},p_j)$&\\\hline
\multirow{4}{*}{Full propositional team logic (\MPT)}&$p_i,\neg p_i,\bot,\nem,$&\multirow{4}{*}{$\wedge,\sor,\nsor,\vee$}\\
&$p_{i_1}\dots p_{i_k} \perp p_{j_1}\dots p_{j_m}$&\\
&$\dep(p_{i_1}\dots p_{i_k},p_j)$&\\
&$p_{i_1}\dots p_{i_k}\subseteq p_{j_1}\dots p_{j_k}$&\\\hline
\end{tabular}
\end{center}
}
\label{tab:logics}
\end{table}%

We are interested in fragments of \MPT that are expressively complete for some nice sets of team properties as defined in \Cref{team_prop_df}. The languages of these fragments are determined  in terms of which atoms and connective are allowed. 
Table \ref{tab:logics} defines the sets of atoms and connectives of the languages of these logics. Apart from \MPT, we consider six other types of propositional logics, namely, classical logic, independence logic, team logic, union closed logic, inclusion logic and dependence logic, each of which has two variants, a weak version that has the empty team property, and a strong version that contains the \nem atom or the linear implication $\multimap$ in its language. The propositional independence logic, propositional inclusion logic and propositional dependence logic we define here are propositional version of their first-order counterparts introduced in \cite{D_Ind_GV,Pietro_I/E,Van07dl}. The propositional team logic defined as in  \Cref{tab:logics} does not directly correspond to the propositional fragment of the first-order team logic studied in \cite{Van07dl,KontinenNurmi2011}, and the reader should not confuse the two logics. The propositional union closed logic is a new logic that was not previously considered in the literature.

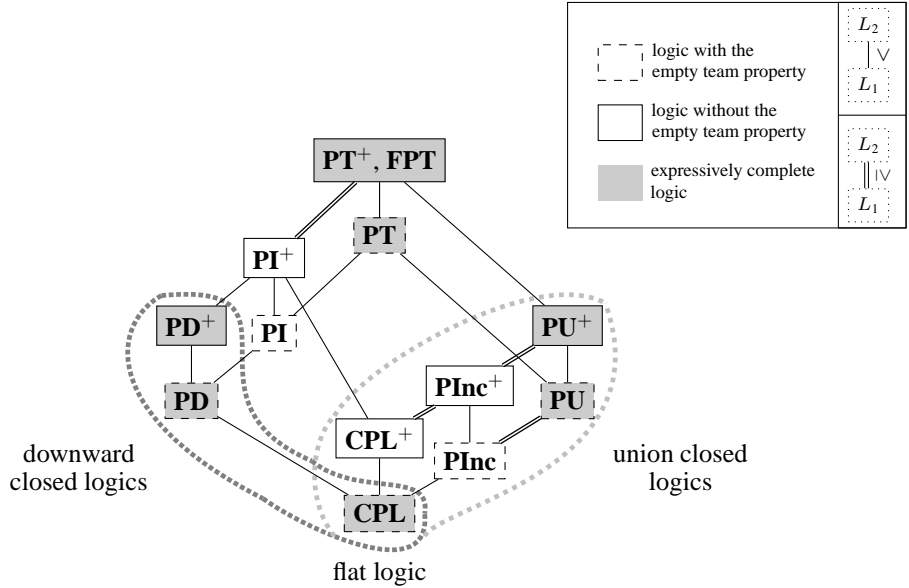
\begin{figure}[t]
\begin{center}
\begin{tikzpicture}[transform shape, scale=1]


\draw[densely dotted, line width=1.8pt, color=gray] (-2.9,2.9) .. controls (-0.8,3.1) and (-2.8,1.5) .. (-0.9,0.7);
\draw[densely dotted, line width=1.8pt, color=gray] (-0.9,0.7) .. controls (-0.3,0.2) and (0.7,0.6) .. (0.6, 0);
\draw[densely dotted,  line width=1.8pt, color=gray] (-2.9,2.9) .. controls (-3.6,2.5) and (-3.7,1.3) .. (-1.55, 0.2);
\draw[densely dotted,  line width=1.8pt, color=gray] (-1.55,0.2) .. controls (-1,-0.2) and (0.7,-1) .. (0.6, 0);

\draw[dotted, line width=1.8pt, color=gray!50] (-0.6,-0.3) .. controls (-1,0.2) and (-0.9,1) .. (-0.6,1.3);
\draw[dotted, line width=1.8pt, color=gray!50] (-0.6,1.3) .. controls (-0.2,2) and (2.1,3.3) .. (3, 2.8);
\draw[dotted, line width=1.8pt, color=gray!50] (0.6,-0.3) .. controls (0.8,0.1) and (3.8,0.6) .. (3, 2.8);

\node[draw, dashed, fill=gray!40] (pd) at (-2.5,1.5) {\PD};
\node[draw,  fill=gray!40] (pds) at (-2.5,2.5) {\PDs};
\node[draw, fill=gray!40] (pu) at (2.5,2.5) {\PU};
\node[draw, dashed, fill=gray!40] (puw) at ( 2.5,1.5) {\PUw};

\node[draw] (pincs) at (1.2,1.7) {\PIncs};
\node[draw, dashed] (pinc) at ( 1.2,0.7) {\PInc};

\node[draw] (ctl) at ( 0,1) {\ECL};
\node[draw] (pis) at (-1.4,3.4) {\PInem};
\node[draw, dashed] (pi) at (-1.4,2.4) {\PInd};
\node[draw, dashed, fill=gray!40] (cpl) at ( 0,0) {\CPC};
\node[draw, fill=gray!40] (pt) at (0,4.7) {\PT, \MPT};
\node[draw, dashed, fill=gray!40] (ptw) at (0,3.7) {\PTw};
\draw (ptw) -- (pt);
\draw (pinc) -- (pincs);
\draw (pt) -- (ptw);
\draw (pd) -- (pi);
\draw (pi) -- (pis);
\draw (pi) -- (ptw);
\draw (cpl) -- (pinc);
\draw (cpl) -- (pd);
\draw (puw) -- (pu);

\draw[double, line width=0.5pt] (pincs) -- (pu);
\draw[double, line width=0.5pt] (pinc) -- (puw);

\draw[double, line width=0.5pt] (ctl) -- (pincs);
\draw (puw) -- (ptw);
\draw (pu) -- (pt);
\draw (ctl) -- (pis);
\draw[double, line width=0.5pt] (pis) -- (pt);
\draw (pds) -- (pis);
\draw (pd) -- (pds);
\draw (cpl) -- (ctl);

\node (fl) at (0,-0.8) {flat logic};




\node at (-4,0.8) {downward}; 
\node at (-4,0.4) {closed logics}; 


\node at (4,0.8) {union closed}; 
\node at (4,0.4) {logics}; 

\node[draw, dashed] at (3.2,6) {\textcolor{white}{12}};
\node at (4.3,6.15) {{\scriptsize logic with the}};
\node at (4.65,5.85) {{\scriptsize empty team property}};

\node[draw] at (3.2,5.2) {\textcolor{white}{12}};
\node at (4.45,5.35) {{\scriptsize logic without the}};
\node at (4.65,5.05) {{\scriptsize empty team property}};

\node[fill=gray!40] at (3.2,4.4) {\textcolor{gray!40}{12}};
\node at (4.7,4.55) {{\scriptsize expressively complete}};
\node at (3.85,4.25) {{\scriptsize logic}};

\node[draw, dotted] (l1a) at (6.5,6.5) {{\scriptsize$L_2$}};
\node[draw, dotted] (l2a) at (6.5,5.7) {{\scriptsize$L_1$}};
\path[draw] (l1a) -- (l2a);

\node at (6.7,6.1) {\rotatebox{90}{{\scriptsize$<$}}};


\node[draw, dotted] (l1b) at (6.5,4.9) {{\scriptsize$L_2$}};
\node[draw, dotted] (l2b) at (6.5,4.1) {{\scriptsize$L_1$}};
\path[draw] (l1b) -- (l2b);
\path[draw] (6.45,4.33) -- (6.45,4.67);

\node at (6.7,4.5) {\rotatebox{270}{{\scriptsize$\geq$}}};


\draw (2.5,3.8) rectangle (7,6.8);

\draw (6.1,3.8) rectangle (7,6.8);
\path[draw] (6.1,5.3) -- (7,5.3);

\end{tikzpicture}
\caption{Expressive power of propositional team logics}
\label{fig:exp_pw}
\end{center}
\end{figure}

\section{Expressive Power and Normal Forms}\label{pisec:nf}

In this section, we study the expressive power of the propositional team logics defined in \Cref{tab:logics} and their normal forms. We will also prove the interconnections illustrated in \Cref{fig:exp_pw}.

To begin with, recall from \cite{VY_PD} that for each team $X$ on a finite set $N=\{i_1,\dots,i_n\}$ of indices, the classical formula 
\[\Theta_X=\bigsor_{s\in X}(p_{i_1}^{s(i_1)}\wedge\dots\wedge p_{i_{n}}^{s(i_{n})})\] 
(in disjunctive normal form)  defines the team $X$ \emph{modulo  subteams},
where $p_i^{1}:=p_i$ and $p_i^0:=\neg p_i$ for any index $i$. That is, for any team $Y$ on $N$,
\begin{equation}\label{ThetaX_prop}
Y\models\Theta_X\iff Y\subseteq X.
\end{equation}
Using this formula, we can prove the expressive completeness of \CPL under the team semantics  for the set of all flat team properties, as illustrated in \Cref{fig:exp_pw}. 
 
 \begin{theorem}\label{cpl_exp_compl}
\CPL characterizes the set of all flat team properties.
\end{theorem}
\begin{proof}
Recall from \Cref{expressive_comp_df,team_prop_df} that we need to show that for any finite set $N=\{i_1,\dots,i_n\}$ of indices, every classical formula $\phi(p_{i_1},\dots,p_{i_n})$ defines a flat team property, i.e., $\llbracket \phi\rrbracket$  is flat, and every flat team property $\mathsf{P}\subseteq\mathcal{P}(2^N)$ is definable by some classical formula $\phi(p_{i_1},\dots,p_{i_n})$, i.e., $\llbracket \phi\rrbracket=\mathsf{P}$. The former follows from \Cref{cpl_prop}. We now prove the latter. Let $\mathsf{P}$ be a flat team property and 
$\{s_1\},\dots,\{s_k\}$ all singleton teams  in $\mathsf{P}$. Putting $X=\{s_1,\dots,s_k\}$, we have $\mathsf{P}=\llbracket\Theta_X \rrbracket$, since for any team $Y$ on $N$,
\[Y\models\Theta_X\iff  Y\subseteq X\in \mathsf{P}\iff Y\in \mathsf{P}.\]
\end{proof}


Clearly, from the above theorem it follows that the standard disjunctive normal form of \CPL under the usual single valuation semantics is also a normal form of \CPL under the team semantics.
Another immediate corollary of the expressive completeness of \CPL for the set of all flat properties is that flatness is equivalent to being logically equivalent to a classical formula in propositional team logics, as stated in the theorem below. Flatness was originally introduced by Hodges \cite{Hodges1997b} in the first-order context, and further studied in  \cite{Van07dl}. In the first-order dependence logic case all classical first-order formulas have the flatness property (see, e.g., \cite{Van07dl}) but the converse  is not true. For example, all first-order sentences (i.e., formulas without free variables) have the flatness property for the trivial reason that their truth in any model is decided by the truth in the singleton team $\{\emptyset\}$ of the empty assignment $\emptyset$ alone (see Lemma 1.1.14 in \cite{Yang_dissertation}). 


 \begin{theorem}\label{cpl_char_prop}
Let $\phi$ be a formula in the language of \MPT. The following are equivalent.
\begin{description}
\item[(i)] $\phi$ has the flatness property.
\item[(ii)] $\phi$ is semantically equivalent to a classical formula.
\item[(iii)] $\phi$ has the downward closure property and the union closure property.
\end{description}
\end{theorem}
 \begin{proof}
 The equivalence of (i) and (ii) follows from \Cref{cpl_exp_compl} and the equivalence of (i) and (iii) is easy to verify.
 \end{proof}
 

Our characterization of classical propositional logic can be equivalently formulated as follows: The classical propositional logic cannot be extended in the context of team semantics to a propositional logic which satisfies both the downward closure and the union closure property. This is remotely reminiscent of the characterization of classical first-order logic, known as Lindstr\"om's Theorem \cite{Lindstrom69}, to the effect that classical first-order logic cannot be extended in the context of abstract logics to a logic which satisfies both the Downward L\"owenheim-Skolem Theorem and the Compactness Theorem.

 
Recall  from  \cite{VY_PD}  that \PD is expressively complete for the set of all downward closed team properties which contain the empty team. Having established also the expressive completeness of \CPL for the set of all flat team properties, we now proceed towards the proof of the other expressive completeness results illustrated in \Cref{fig:exp_pw}.

The set of subteams  of a fixed team $X$ on a finite set $N$ is characterized by the classical formula $\Theta_X$ (in the sense of (\ref{ThetaX_prop})). We now show that with the non-emptiness or the nonempty disjunction, a team can  be characterized precisely. 

\begin{lemma}\label{X_ThetaX_nem}
Let $X$ be a team on a finite set $N=\{i_1,\dots,i_{n}\}$ of indices. Define 
\[\Theta^{\ast}_X:=\bigsor_{s\in X}(p_{i_1}^{s(i_1)}\wedge\dots\wedge p_{i_{n}}^{s(i_{n})}\wedge\nem)\quad\text{and}\quad\Theta^{\ast\ast}_X:=\bignsor_{s\in X}(p_{i_1}^{s(i_1)}\wedge\dots\wedge p_{i_{n}}^{s(i_{n})}).\]
For any  team $Y$ on $N$, we have
\begin{description}
\item[(i)] \(Y\models \Theta^\ast_X\iff Y= X.\)
\item[(ii)] \(Y\models \Theta^{\ast\ast}_X\iff Y= X\) or $Y=\emptyset$.
\end{description}
\end{lemma}
\begin{proof}
The direction ``$\Longleftarrow$'' for both items is obvious. For the direction ``$\Longrightarrow$'' of   item (i), suppose $Y\models \Theta^\ast_X$. If $X=\emptyset$, then $\Theta^\ast_X= \bot$, hence $Y=\emptyset=X$. Otherwise, for each $s\in X$ there exists a set $Y_s$ such that 
\[Y=\bigcup_{s\in X}Y_s\text{ and }Y_s\models p_{i_1}^{s(i_1)}\wedge\dots\wedge p_{i_n}^{s(i_n)}\wedge\nem.\]
Clearly, $Y_s=\{s\}$ for each $s\in X$, implying $Y= X$.

For the direction ``$\Longrightarrow$'' of  item (ii), suppose $Y\models \Theta^{\ast\ast}_X$. If $Y=\emptyset$, then we are done. Otherwise, $Y=X$ is proved by a similar argument to  that of item (i).
\end{proof}



Next, we prove that many of the logics we defined in Section 2 (see \Cref{tab:logics}) are expressively complete for certain sets of team properties (see \Cref{expressive_comp_df,team_prop_df} for definitions of the relevant notions).


\begin{theorem}\label{PT_exp_pw} 
\begin{description}
\item[(i)] Both \PT and \MPT  characterize the set of all team properties, i.e., every team property is definable in the logics. In particular, $\PT\equiv \MPT$.
\item[(ii)] \PTw characterizes the set of all team properties which contain the empty team.
\item[(iii)] \PUw characterizes the set of all union closed team properties  which contain the empty team.
\item[(iv)] \PU characterizes the set of all union closed team properties.
\item[(v)] \PD characterizes the set of all downward closed team properties  which contain the empty team. 
\item[(vi)] $\PD^+$ characterizes the set of all downward closed team properties.
\end{description}
\end{theorem}
\begin{proof}
Let $N=\{i_1,\dots,i_n\}$ be an arbitrary finite set of indices. 

(i) Obviously every formula $\phi(p_{i_1},\dots,p_{i_n})$ in the language of \PT or \MPT defines a team property on $N$, i.e., $\llbracket \phi\rrbracket\subseteq \mathcal{P}(2^N)$. Conversely, for any team property  $\mathsf{P}\subseteq \mathcal{P}(2^N)$, we shall show  $\mathsf{P}=\llbracket \bigbor_{X\in \mathsf{P}}\Theta^\ast_X\rrbracket$. 

By \Cref{X_ThetaX_nem}(i), for any team $Y$ on $N$, 
\[Y\models\bigbor_{X\in \mathsf{P}}\Theta^\ast_X\iff \exists X\in \mathsf{P}(Y= X)\iff Y\in \mathsf{P}.\]
 In particular, if $\mathsf{P}=\emptyset$, then $\llbracket \bigbor_{X\in \emptyset}\Theta^\ast_X\rrbracket=\llbracket\bot\wedge\nem\rrbracket=\emptyset$.
 
(ii) Since formulas $\phi(p_{i_1},\dots,p_{i_n})$  in the language of \PTw have the empty team property, we have $\emptyset\in\llbracket \phi\rrbracket\subseteq \mathcal{P}(2^N)$. Conversely, for any  team property  $\mathsf{P}\subseteq \mathcal{P}(2^N)$ with $\emptyset\in \mathsf{P}$, we show $\mathsf{P}=\llbracket \bigbor_{X\in \mathsf{P}}\Theta^{\ast\ast}_X\rrbracket$. 

By \Cref{X_ThetaX_nem}(ii), for any team $Y$ on $N$, 
\[Y\models\bigbor_{X\in \mathsf{P}}\Theta^{\ast\ast}_X\iff \exists X\in \mathsf{P}(Y= X)\text{ or }Y=\emptyset\iff Y\in \mathsf{P}.\]

(iii) It is easy to show by induction that every formula $\phi$ in the language of \PUw has the union closure property and the empty team property, which imply that $\llbracket \phi\rrbracket$ is a union closed team property that contains the empty team. Conversely, for any  union closed team property  $\mathsf{P}\subseteq \mathcal{P}(2^N)$ with $\emptyset\in\mathsf{P}$, we show that $\mathsf{P}=\llbracket \bigsor_{X\in \mathsf{P}}\Theta^{\ast\ast}_X\rrbracket$.

 If $Y\in \mathsf{P}$, then, by \Cref{X_ThetaX_nem}(ii) we have  $Y\models\Theta^{\ast\ast}_Y$ and $Y\models \bigsor_{X\in \mathsf{P}}\Theta^{\ast\ast}_X$. Conversely, if $Y\models \bigsor_{X\in \mathsf{P}}\Theta^{\ast\ast}_X$, then for each $X\in\mathsf{P}$ there is $Y_X\subseteq Y$ such that $Y=\bigcup_{X\in\mathsf{P}}Y_X$ and $Y_X\models \Theta^{\ast\ast}_X$. By  \Cref{X_ThetaX_nem}(ii), we have $Y_X=\emptyset$ or $Y_X=X$ for each $X\in \mathsf{P}$. Since $\emptyset\in \mathsf{P}$ and $\mathsf{P}$ is a union closed team property, we conclude that $Y\in \mathsf{P}$.

(iv) 
Obviously $\llbracket\nem\rrbracket$ is a union closed team property. Thus, by item (iii), for every formula $\phi$ in the language of \PU, $\llbracket \phi\rrbracket$ is a  union closed team property. Conversely, for any union closed team property  $\mathsf{P}\subseteq \mathcal{P}(2^N)$, we show that  $\mathsf{P}$ is definable by some formula $\phi$ in the language of \PU. 

If $\mathsf{P}$ contains the empty team, then  $\mathsf{P}=\llbracket\bigsor_{X\in \mathsf{P}}\Theta^{\ast\ast}_X\rrbracket$ by item (iii). If $\emptyset\notin \mathsf{P}$, then it is easy to verify that $\mathsf{P}=\llbracket \nem\wedge\bigsor_{X\in \mathsf{P}}\Theta^{\ast\ast}_X\rrbracket$.

(v) This item is a consequence of results in \cite{VY_PD}. Note that in \cite{VY_PD}, propositional dependence logic and some of its variants (including propositional inquisitive logic) are all shown to be expressively complete for the set of all downward closed team properties which contain the empty team. With a slight abuse of notation, we denote in this paper by \PD any of these equivalent logics.


(vi) It is easy to show, by induction, that every formula $\phi$ in the language of $\PD^+$ has the downward closure property, which implies that $\llbracket \phi\rrbracket$ is a downward closed team property. Conversely, for any downward closed team property  $\mathsf{P}\subseteq \mathcal{P}(2^N)$, we show that  $\mathsf{P}$ is definable by some formula in the language of \PU. 

If $\mathsf{P}$ is a downward closed team property that contains the empty set, then by item (v), we know that  $\mathsf{P}$ is definable by some formula in the language of \PD (thus also in the language of $\PD^+$). If $\emptyset\notin \mathsf{P}$, then since $ \mathsf{P}$ is a downward closed team property, we must have $ \mathsf{P}=\emptyset$. Clearly, $\mathsf{P}=\emptyset=\llbracket \top\multimap\bot\rrbracket$.
\end{proof}

Results in \Cref{PT_exp_pw,cpl_exp_compl} are illustrated in \Cref{fig:exp_pw}, where all those expressively complete logics  are represented in shaded rectangles labeled with their corresponding characteristic team properties. The logics \MPT, \PT and \PTw do not have a label in  \Cref{fig:exp_pw}, as they characterize the set of arbitrary team properties (with or without the empty team). Except for the flat team property, each characteristic team property we study here has two variants. One with the empty team in the property  and the other one without this constraint. We have an  expressively complete logic for certain set of  team properties that has the empty team property (indicated by a solid rectangle) and an expressively complete logic for the same set of team properties without the empty team property (indicated by a dashed rectangle).



We remarked in Section~\ref{premin} that in any propositional team logic with the empty team property, the nonempty disjunction $\nsor$ and the might modality $\triangledown$ are inter-definable (Equation (\ref{nsor_modality_interdf})). Consequently, classical propositional logic extended with the might modality $\triangledown$ has the same expressive power as \PUw and thereby is also expressively complete for the set of all union closed team properties which contain the empty team. 


Propositional independence logic \PI is expressively properly included in the expressively strongest logic \PT. In particular, the independence atoms are definable in \PT: 
\begin{equation}\label{indatm_df}
p_{i_1}\dots p_{i_k} \perp p_{j_1}\dots p_{j_m}\equiv\bigvee_{X\in \mathcal{X}_{I,J}}\Theta_{X}^\ast,
\end{equation}
 where $I=\{i_1,\dots,i_k\}$, $J=\{j_1,\dots,j_m\}$ and
\[\mathcal{X}_{I,J}=\{X\subseteq \{0,1\}^{I\cup J}\mid X=(X\upharpoonright I)\times (X\upharpoonright J)\}.\] 
The expressive power of \PI is an open problem.  

Another immediate consequence of the expressive completeness of \PT for the set of all team properties is that  all the possible atoms and all the instances of all the possible connectives are expressible in the expressively strongest logic  \PT. Consider the \emph{Boolean  negation}  $\sim$ defined as 
\begin{itemize}
\item $X\models\sim\!\phi$ iff $X\not\models\phi$
\end{itemize}
Clearly $\nem\equiv \sim\!\bot$ and $\phi\vee\psi\equiv\sim(\sim\phi\wedge\sim\psi)$. Thus  \ECL extended with the Boolean negation $\sim$ has the same expressive power as \PT. See \cite{Luck2016} for other properties of the Boolean negation $\sim$ and a complete axiomatization for propositional dependence logic extended with $\sim$. In this paper we will restrict our attention  to the considerably simpler logical constant \nem instead of $\sim$. Note that from the equivalence of the two logics, we only derive that \emph{every instance} of $\sim \phi$ is expressible in \PT, but Boolean negation turns out to be not \emph{uniformly definable} in \PT (see \cite{Yang15uniform}).

From the proofs of \Cref{cpl_exp_compl} and Theorem \ref{PT_exp_pw} we obtain interesting disjunctive normal forms for the logics, as listed in (the self-explanatory) \Cref{tb:nf}. 
It is worth taking note of the many similarities and at the same time the subtle differences of these normal forms.


\begin{table}[t]
\begin{center}
\begin{tabular}{|l|c|}
\hline
Logic&Normal Form\\\hline
\PT& \multirow{2}{*}{$\displaystyle\bigbor_{f\in F}\bigsor_{s\in X_f}(p_{i_1}^{s(i_1)}\wedge\dots\wedge p_{i_{n}}^{s(i_{n})}\wedge\nem)$}\\
 \MPT&\\\hline
\PTw& $\displaystyle\bigbor_{f\in F}\bignsor_{s\in X_f}(p_{i_1}^{s(i_1)}\wedge\dots\wedge p_{i_{n}}^{s(i_{n})})$\\\hline
\multirow{2}{*}{\PU}& either $\displaystyle\nem\wedge\bigsor_{f\in F}\bignsor_{s\in X_f}(p_{i_1}^{s(i_1)}\wedge\dots\wedge p_{i_{n}}^{s(i_{n})})$ or $\displaystyle\bigsor_{f\in F}\bignsor_{s\in X_f}(p_{i_1}^{s(i_1)}\wedge\dots\wedge p_{i_{n}}^{s(i_{n})})$\\\hline
\PUw& $\displaystyle\bigsor_{f\in F}\bignsor_{s\in X_f}(p_{i_1}^{s(i_1)}\wedge\dots\wedge p_{i_{n}}^{s(i_{n})})$\\\hline
\PD& $\displaystyle\bigbor_{f\in F}\bigsor_{s\in X_f}(p_{i_1}^{s(i_1)}\wedge\dots\wedge p_{i_{n}}^{s(i_{n})})$\\\hline
$\PD^+$& either $\displaystyle\bigbor_{f\in F}\bigsor_{s\in X_f}(p_{i_1}^{s(i_1)}\wedge\dots\wedge p_{i_{n}}^{s(i_{n})})$ or $\top\multimap\bot$ \\\hline
\CPL& $\displaystyle\bigsor_{s\in X_f}(p_{i_1}^{s(i_1)}\wedge\dots\wedge p_{i_{n}}^{s(i_{n})})$\\\hline
\end{tabular}
\end{center}
\caption{Normal forms of propositional team logics}
\label{tb:nf}
\end{table}%

We now prove the results in \Cref{fig:exp_pw} concerning the comparison of the logics in terms of their expressive powers. In \Cref{fig:exp_pw}, the logics placed in the same rectangle have the same expressive power. If a line connects two sets of logics, then the logics $\LL_2$ positioned above are expressively strictly stronger than the logics $\LL_1$ positioned below, i.e., $\LL_1< \LL_2$. If instead, $\LL_1$ and $\LL_2$ are connected by a double line, then only $\LL_1\leq \LL_2$ is known. As discussed already, the logic \PT and \MPT are expressively complete for the set of all team properties, thus they are both the expressively strongest logics. The logic \PTw is expressively complete for the set of all team properties which contain the empty set, therefore it has stronger expressive power than the two logics \PUw and \PI that characterize certain team properties which contain the empty set. For any other pair of logics that are linked by a solid line or a double line in  \Cref{fig:exp_pw}, if the one that is positioned above is an extension of the other, then it obviously has stronger expressive power. The logics \CPL, \PUw, \PD, \PI and \PTw that have the empty team property are strictly weaker than their corresponding logics \ECL, \PU, \PDs, \PInem and \PT that do not have the empty team property, respectively. \PD is strictly stronger than \CPL because classical formulas have the flatness property, while \PD has formulas that lack the flatness property (see \cite{VY_PD} for detail). The union closed logics \ECL, \PUw and \PU are expressively different from the non-union closed logics \PInem, \PTw and \PT, respectively. For instance, the formulas $p_0\perp p_1$ and $p_0\vee p_1$ are not closed under unions. Similarly, the downward closed logics \PD and \PDs are expressively different from the non-downward closed logics \PI and \PInem, respectively. For instance, the formula $p_0\perp p_1$  is not closed downward.  

It now remains to prove that \PD and \PDs are expressively weaker than \PI and \PInem, respectively. This reduces to showing that dependence atoms are expressible in \PI and in \PInem. First, note that the independence atoms $\vec{x}\perp \vec{y}$ are known as \emph{unconditional} independence atoms in the literature of independence logic. A \emph{conditional} independence atom \cite{D_Ind_GV} is written as $\vec{x} \perp_{\vec{z}} \vec{y}$ and its semantics is defined by the clause
\begin{itemize}
\item $X\models\, p_{j_1}\dots p_{j_b} \perp_{p_{i_1}\dots p_{i_a} } p_{k_1}\dots p_{k_c}$ iff for all $s,s'\in X$ with $s(\vec{i})=s'(\vec{i})$, there exists $s''\in X$ such that
\[s''(\vec{i})=s(\vec{i})=s'(\vec{i}),\quad s''(\vec{j})=s(\vec{j})\quad\text{and}\quad s''(\vec{k})=s'(\vec{k}),\]
where $\vec{i}=i_1\dots i_a$, $\vec{j}=j_1\dots j_b$ and $\vec{k}=k_1\dots k_c$.
\end{itemize}
In our setting conditional independence atoms are expressible in terms of unconditional ones:
\[\displaystyle \vec{x} \perp_{p_{i_1}\dots p_{i_a}} \vec{y}\equiv\bigsor_{s\in 2^I}\big(p_{i_1}^{s(i_1)}\wedge\dots\wedge p_{i_a}^{s(i_a)}\wedge( \vec{x} \perp \vec{y})\big),\]
where $I=\{i_1,\dots,i_a\}$. As observed already in \cite{D_Ind_GV} in the context of first-order logic, dependence atoms are definable in terms of conditional independence atoms: 
\[\dep(\vec{x},p_{i})\equiv p_{i}\perp_{\vec{x}}p_{i}.\]
Putting these altogether, we conclude that dependence atoms are expressible in \PI and in \PInem. This completes the proof of the inter-relationships shown in \Cref{fig:exp_pw}.

%

\section{Metalogical properties}\label{sec:meta_prop}


In this section, we study metalogical properties of propositional team logics. We will prove that propositional team logics are compact, and  closed under classical substitutions.


\subsection{Compactness}


Propositional inquisitive logic was shown to be compact in  \cite{InquiLog}. As a consequence, propositional logics of dependence  are all compact (see Theorem 3.3 in \cite{VY_PD}). Moreover, from the same argument as given in \cite{ivano_msc} the compactness of all the other propositional team logics follows as well.  Below we present a  sketch of this proof (which makes essential use of K\"{o}nig's Lemma) using the terminologies of this paper for the benefit  of the reader.

\begin{theorem}[Compactness Theorem]\label{compactness_PI}
For any set $\Gamma\cup\{\phi\}$ of formulas  in the language of an arbitrary propositional team logic, if $\Gamma\models\phi$, then there exists a finite set $\Gamma_0\subseteq\Gamma$ such that $\Gamma_0\models\phi$. 
\end{theorem}
\begin{proof}

Let $\Gamma=\{\theta_k\mid k\in\mathbb{N}\}$. For each $k\in\mathbb{N}$, define $\gamma_k=\theta_1\wedge\dots\wedge\theta_k$. It is sufficient to show that $\gamma_k\models\phi$ for some $k\in\mathbb{N}$. Towards a contradiction, assume otherwise. Then for each $k\in\mathbb{N}$, there exists a team $X\subseteq 2^{\mathbb{N}}$ such that $X\models\gamma_k$ and $X\not\models\phi$. Let $N_k$ be the set of all indices of all propositional variables occurring in $\gamma_k$ and $\phi$. By the locality property of propositional team logics, the finite set 
\[\mathcal{X}_k=\{X\subseteq 2^{N_k}:X\models\gamma_k\text{ and }X\not\models\phi\}\] 
is nonempty. Put $T=\{\emptyset\}\cup\bigcup_{k\in\mathbb{N}}\mathcal{X}_k$. Define a relation $\leq$ on $T$ by putting 
\begin{itemize}
\item $\emptyset\leq X$ for all $X\in T$;
\item $X\leq Y$ iff $Y\upharpoonright \rm{dom}(X)=X$.
\end{itemize}
It is not hard to see that $(T,\leq)$ is a finitely branching infinite tree. By K\"{o}nig's Lemma, the tree has an infinite branch $\langle X_k\mid k\in\mathbb{N}\rangle$, where $X_k\in\mathcal{X}_k$ for each $k\in\mathbb{N}$. Putting $N=\bigcup_{k\in\mathbb{N}}N_k$, this infinite branch determines a team $X\subseteq 2^N$ in such a way that $X\upharpoonright N_k=X_k$ for each $k\in\mathbb{N}$. Clearly, $X\not\models\phi$ and $X\models \gamma_k$ for each $k\in\mathbb{N}$. These contradict $\Gamma\models\phi$. 
\end{proof}

\subsection{Closure under classical substitutions}






A \emph{substitution} of a propositional team logic \LL is a mapping $\sigma$ from the set ${\rm Form}_\LL$ of all well-formed formulas of \LL into the set ${\rm Form}_\LL$ itself that commutes with the connectives and atoms of \LL. We say that \LL is \emph{closed under} the substitution $\sigma$,
if for any set $\Gamma\cup\{\phi\}$ of formulas of \LL,
\[\Gamma\models\phi\Longrightarrow \{\sigma(\gamma)\mid \gamma\in \Gamma\}\models\sigma(\phi).\]
If \LL is closed under all substitutions, then we say that \LL is \emph{closed under uniform substitution}.
 The logics \MPT, \PTw, \PT, \PD, $\PD^+$, \PI and \PInem  are not closed under uniform substitution, because, for instance, we have 
\[p_i\sor p_i\models p_i\text{ and }p_i\nsor p_i\models p_i,\] 
whereas 
\[\dep(p_i)\sor \dep(p_i)\not\models \dep(p_i),\quad(p_i\perp p_i)\sor (p_i\perp p_i)\not\models p_i\perp p_i,\]
\[(p_i\vee\neg p_i)\sor(p_i\vee\neg p_i)\not\models(p_i\vee\neg p_i)\text{ and }(p_i\vee\neg p_i)\nsor(p_i\vee\neg p_i)\not\models(p_i\vee\neg p_i).\]
The logics \PU and \PUw are not closed under uniform substitution either (for at least the trivial reason that strings of the form $\neg (\phi\nsor\psi)$ are not well-formed formulas), but  nontrivial counter-examples of the above kind for the two logics are yet to be found. 
It was shown in  \cite{ivano_msc} and \cite{IemhoffYang15} that propositional logics of dependence are, nevertheless, closed under \emph{flat substitutions}, i.e., substitutions $\sigma$ such that $\sigma(p)$ has the flatness property for any propositional variable $p$. Using the method in \cite{IemhoffYang15}, we will prove in this section that  propositional team logics are closed under \emph{classical substitutions}, i.e., substitutions $\sigma$ such that $\sigma(p)$ is a classical formula (i.e., a formula in the language of \CPL) for any propositional variable $p$.

Let us start by examining in detail the notion of substitution in our logics.
The well-formed formulas of the propositional team logics we consider in this paper are assumed to be in \emph{negation normal form} and we do not allow arbitrary formulas to occur in a dependence or independence atom. Strings of the form $\neg\phi$, $\vec{\phi}\perp\vec{\psi}$, $\dep(\vec{\phi},\psi)$ and $\vec{\phi}\subseteq\vec{\psi}$ are not necessarily well-formed formulas of our logics. As such, the notion of substitution is actually not well-defined in our logics. 
To derive our intended closure under substitution result we will then first need to seek for ways to make sense of the notion of substitution in our logics. In general, there are two possible solutions: either to expand the languages of the logics so as to allow more well-formed formulas, or to restrict the range of a substitution to a subset of the full set $ {\rm Form}_\LL$ of well-formed formulas. We will take both approaches at the same time.  
We will confine ourselves to classical substitutions only and will also expand the language of our logics to include every substitution instance $\sigma(\phi)$ of a classical substitution $\sigma$ to be a well-formed formula. Our reason for restricting attention to classical substitutions only is twofold. Conceptually, we do not have a good intuition of the intended semantics of the formulas $\neg(\phi\nsor\psi)$ and $\neg(\phi\vee\psi)$ or of the dependence and independence atoms $\vec{\phi}\perp\vec{\psi}$, $\dep(\vec{\phi},\psi)$ and $\vec{\phi}\subseteq\vec{\psi}$ with arbitrary arguments. Technically, arbitrary substitutions are not very interesting, as the logics are not closed under uniform substitution.


For simplicity, in what follows we will only work with \MPT which has the maximal set of atoms and connectives among the propositional team logics we consider in this paper. Similar results for the other logics can be easily obtained as corollaries of those for \MPT. Let us now expand the language of \MPT and include strings of the forms $\neg\alpha$, $\top$, $\vec{\alpha}\perp\vec{\beta}$, $\dep(\vec{\alpha},\beta)$ and $\vec{\alpha}\subseteq\vec{\beta}$ as well-formed formulas, where $\alpha,\beta,\vec{\alpha}$ and $\vec{\beta}$ are classical formulas or sequences of classical formulas. Denote the extended logic by $\mathbb{MPT}$.
Observe that $\mathbb{MPT}$ will, clearly, have the same expressive power as \MPT, since the latter is already expressively complete for the set of all team properties (\Cref{PT_exp_pw}(i)).

We now define the semantics of $\mathbb{MPT}$. Given sequences $\vec{\alpha}=\alpha_1\dots\alpha_k$ and $\vec{\beta}=\beta_1\dots\beta_k$  of classical formulas,  we define an equivalence relation $\sim_{(\vec{\alpha},\vec{\beta})}$ on the set of all valuations as follows:
\[
 s \sim_{(\vec{\alpha},\vec{\beta})} s'\quad \text{ iff } \quad\forall  i\in\{1,\dots, k\} \, (\{s\}\models \alpha_i \iff \{s'\} \models \beta_i).
 \]
 We  write $\sim_{\vec{\alpha}}$ for $\sim_{(\vec{\alpha},\vec{\alpha})}$.
 
 \begin{definition}\label{ext_mpt_ts}
The team semantics of well-formed formulas of $\mathbb{MPT}$ is defined inductively in the same way as in \Cref{TS_CT,TS_PT} and additionally we have the following extra clauses:
\begin{itemize}
\item $X\models \neg\alpha$ iff $s\not\models\alpha$ in the usual sense for all $s\in X$.
\item $X\models\top$ always holds.
\item $X\models\vec{\alpha}\perp\vec{\beta}$ iff for all $s,s'\in X$, there exists $s''\in X$ such that
\(s\sim_{\vec{\alpha}}s''\text{ and } \mathop{s'\sim_{\vec{\beta}}s''}.\)
\item $X\models\dep(\vec{\alpha},\beta)$ iff for all $s,s'\in X$, 
\(\text{if }s\sim_{\vec{\alpha}}s'\text{ then }s\sim_{\beta}s'.\)
\item $X\models\vec{\alpha}\subseteq\vec{\beta}$ iff for all $s\in X$, there exists $s'\in X$ such that
\(s\sim_{(\vec{\alpha},\vec{\beta})}s'.\) 
\end{itemize}
\end{definition}

The above definition deserves some comments. In the literature of logics of dependence and independence, negation is usually treated only \emph{syntactically}. That is, a negated classical formula $\neg\phi$ is defined to have the same semantics as the unique formula $\phi^{\sim}$ in \emph{negation normal form} obtained by exhaustively applying the De Morgan's laws and the following syntactic rewrite rules: 
\begin{equation}\label{neg_df}
\begin{array}{rclcrclcrcl}
 p^\sim&\mapsto& \neg p &~~~~~~~~& \top^\sim&\mapsto&\bot&~~~~~~~~& (\phi\wedge\psi)^\sim&\mapsto&\phi^\sim\sor\psi^\sim\\
(\neg p)^\sim&\mapsto& p  &&\bot^\sim&\mapsto&\top& &(\phi\sor\psi)^\sim&\mapsto&\phi^\sim\wedge\psi^\sim \\
 \end{array}
\end{equation}
A routine inductive proof shows that $\neg\alpha\equiv\alpha^\sim$ for all classical formulas $\alpha$ (or see \cite{IemhoffYang15} for the proof), i.e., our negation as defined in \Cref{ext_mpt_ts} coincides with the above syntactic negation when applied to classical formulas. It is also worth noting that our negation corresponds to the defined connective $\sim\downarrow$ in Hodges \cite{Hodges1997a,Hodges1997b}.

The extended inclusion atom $\vec{\alpha}\subseteq\vec{\beta}$ is also studied in \cite{HellaStumpf2015} in the context of modal inclusion logic. There it is shown that the usual modal logic extended with the extended inclusion atoms has the same expressive power as the usual modal logic extended with the might modality $\triangledown$ (see Section 2, under the heading ``non-emptiness atom"). Given this result, it is natural to conjecture  that classical propositional logic extended with the extended inclusion atoms has the same expressive power as classical propositional logic extended with the might modality $\triangledown$, and it is also 
 expressively complete for the set of  union closed team properties which contain the empty team (see the remark after \Cref{PT_exp_pw}).

If $\alpha$ is a classical formula and $\sigma$ is a classical substitution, then $\sigma(\alpha)$ is still a classical formula. In particular, given any classical substitution, the substitution instance of an extended dependence and independence atom is a well-formed formula of the extended logic $\mathbb{MPT}$. Having the notion of classical substitution well-defined in $\mathbb{MPT}$, we will now  prove that $\mathbb{MPT}$ is closed under classical substitutions, namely the following theorem holds.

\begin{theorem}\label{closure_flatsub} Let $\Gamma\cup\{\phi\}$ be a set of formulas in the language of $\mathbb{MPT}$, and $\sigma$ a classical substitution. If $\Gamma\models\phi$, then $\{\sigma(\gamma)\mid \gamma\in \Gamma\}\models\sigma(\phi)$.
\end{theorem}

 Note that this theorem also implies that the original logic \MPT is closed under those  classical substitutions $\sigma$ such that $\sigma(\phi)$ is  a well-formed formula of \MPT, whenever $\phi$ is. In this sense, \Cref{closure_flatsub} also characterizes the behavior of substitutions in the original logic \MPT.

For the proof of \Cref{closure_flatsub}, we first establish the following lemma, which generalizes Lemma 3.5 in \cite{IemhoffYang15} concerning flat substitutions (note that classical substitutions are also flat substitutions). For any valuation $s$ and any substitution $\sigma$, define a valuation $s_\sigma$ as 
\[
 s_\sigma (i) =  
  \left\{
   \begin{array}{ll}
    1 & \text{if $\{s\}\models \sigma(p_i)$;} \\
    0 & \text{if $\{s\}\not\models \sigma(p_i)$.} \\
   \end{array}
  \right.
\]
For any team $X$, we write
\(
 X_\sigma = \{s_\sigma \mid s \in X \}.
\)

\begin{lemma}
 \label{lemsigmaoperation} 
For any formula $\phi$ in the language of $\mathbb{MPT}$ and any classical substitution $\sigma$, we have 
\[
 X \models \sigma(\phi) \iff X_\sigma \models \phi.
\]
\end{lemma}
\begin{proof}

We prove the lemma  by induction on  $\phi$.
The case $\phi=\{\nem,\bot,\top\}$ is trivial. Case $\phi=p_i$. Since $\sigma(p_i)$ has the flatness property, we have 
\[
  X \models \sigma(p_i) \iff \forall s \in X (\{s\}\models \sigma(p_i)) \iff 
  \forall s_\sigma \in X_\sigma (\{s_\sigma\} \models p_i) \iff X_\sigma \models p_i.
\]

Case $\phi=\vec{\alpha}\perp \vec{\beta}$, where $\vec{\alpha}$ and $\vec{\beta}$ are sequences of classical formulas. To show the direction ``$\Longrightarrow$'', assume $X\models \sigma(\vec{\alpha})\perp \sigma(\vec{\beta})$. For any $s_\sigma,s'_\sigma\in X_\sigma$, we have $s,s'\in X$ and there is $s''\in X$ such that $s\sim_{\sigma(\vec{\alpha})}s''$ and $s'\sim_{\sigma(\vec{\beta})}s''$. By the induction hypothesis, we have $s_\sigma\sim_{\vec{\alpha}}s''_\sigma$ and $s'_\sigma\sim_{\vec{\beta}}s''_\sigma$, as required. The other direction ``$\Longleftarrow$'' is proved analogously.

The cases $\phi=\dep(\vec{\alpha},\beta)$ and $\phi=\vec{\alpha}\subseteq \vec{\beta}$, where $\vec{\alpha},\vec{\beta},\beta$ are (sequences of) classical formulas, are proved analogously.

Case $\phi=\psi \nsor \chi$. 
For the direction ``$\Longrightarrow$'', assuming $X \models \sigma(\psi) \nsor \sigma(\chi)$, if $X=\emptyset$, then $X_\sigma=\emptyset\models \psi \nsor \chi$. If $X\neq\emptyset$, then there are nonempty sets $Y,Z \subseteq X$ such that $X = Y \cup Z$ and $Y \models \sigma(\psi)$ and $Z \models \sigma(\chi)$. Since $Y_\sigma \cup Z_\sigma = X_\sigma$ and $Y_\sigma,Z_\sigma\neq\emptyset$, we obtain $X_\sigma\models \psi \nsor \chi$ by the induction hypothesis. The other direction ``$\Longleftarrow$'' is proved analogously.

The case $\phi=\psi \sor \chi$ is proved analogously. The cases $\phi=\neg\alpha$ for $\alpha$ classical, $\phi=\psi\wedge\chi$ and $\phi=\psi\vee\chi$ follow readily from the induction hypothesis. 
\end{proof}


Finally, we give the proof of \Cref{closure_flatsub}.

\begin{proof}[Proof of \Cref{closure_flatsub}]
 If $\Gamma\models\phi$, then for any team $X$,
\begin{align*}
X\models\sigma(\gamma)\text{ for all }\gamma\in\Gamma&\Longrightarrow X_\sigma\models\gamma\text{ for all }\gamma\in\Gamma~~\text{ (by \Cref{lemsigmaoperation})}\\
&\Longrightarrow X_\sigma\models\phi~~\text{ (by the assumption)}\\
&\Longrightarrow X\models\sigma(\phi)~~\text{ (by \Cref{lemsigmaoperation})}
\end{align*}
Hence $\{\sigma(\gamma)\mid \gamma\in \Gamma\}\models\sigma(\phi)$.
\end{proof}

\section{Axiomatizations}\label{axioma}

In this section, we  study the axiomatization problem of propositional team logics. For a set of $n$ propositional variables, there are in total $2^{2^n}$ teams. Therefore propositional team logics are clearly \emph{decidable}. Concrete axiomatizations for propositional logics of dependence can be found in \cite{VY_PD}. In this section, we give natural deduction systems and prove the Completeness Theorem for the logics \PT, \ECL, \PInem and  \PIncs. Among the propositional team logics we have defined, these are the logics that have the non-emptiness \nem but not the nonempty disjunction $\nsor$ in their languages. The problem of finding (nontrivial) axiomatizations for the other logics, especially for propositional independence logic \PI, is open.



%
%

\subsection{\PT}\label{ptlogic}

In this subsection, we define a natural deduction system of strong propositional team logic (\PT), an expressively strongest logic, and prove the Soundness and Completeness Theorems for it.


We first present our natural deduction system. We adopt the standard conventions of natural deduction systems (readers who are not familiar with natural deduction systems are referred to, e.g., \cite{logic_structure_vanDalen,TroelstraSchwichtenberg96_book}). For example, the letter $D$ (with or without subscripts) in the following definition stand for an arbitrary derivation (with or without extra assumptions). As discussed in Section 4.2, \PT (and also the other propositional team logics to be axiomatized) are \emph{not} closed under uniform substitution. Therefore, the (sound and complete) deduction system given by the definition below will \emph{not} admit the \emph{Substitution Rule} 
\[\AxiomC{$\phi(p_{i_1},\dots,p_{i_n})$} \RightLabel{\Sub}\UnaryInfC{$\phi(\sigma(p_{i_1})/p_{i_1},\dots,\sigma(p_{i_n})/p_{i_n})$} \DisplayProof\]
In particular, the axioms and rules of the system presented below using concrete formulas such as $p_i$ should \emph{not} be read as  schemata. Only the metalanguage symbols $\phi$, $\psi$ and $\chi$  stand for arbitrary formulas. 


\begin{definition}[A natural deduction system of \PT]\label{Natrual_Deduct_CTLe}\

\vspace{-1\baselineskip}

\begin{center}
\setlength{\tabcolsep}{6pt}
\renewcommand{\arraystretch}{1.8}
\setlength{\extrarowheight}{1pt}
\begin{tabular}{|C{0.465\linewidth}C{0.465\linewidth}|}
\multicolumn{2}{c}{\textbf{AXIOMS}}\\\hline\hline
Atomic excluded middle&\nem introduction\\
\AxiomC{}\noLine\UnaryInfC{} \RightLabel{\exclmid}\UnaryInfC{$p_i\sor\neg p_i$}\noLine\UnaryInfC{}\DisplayProof&\AxiomC{} \RightLabel{\nemi}\UnaryInfC{$\bot\bor\nem$} \DisplayProof \\\hline
\end{tabular}
\end{center}


\begin{center}
\setlength{\tabcolsep}{6pt}
\renewcommand{\arraystretch}{1.8}
\setlength{\extrarowheight}{1pt}
\begin{tabular}{|C{0.465\linewidth}C{0.465\linewidth}|}
\multicolumn{2}{c}{\textbf{RULES}}\\\hline\hline
Conjunction introduction&Conjunction elimination\\
\AxiomC{}\noLine\UnaryInfC{$D_1$}\noLine\UnaryInfC{$\phi$}\AxiomC{}\noLine\UnaryInfC{$D_2$}\noLine\UnaryInfC{$\psi$}\RightLabel{\ci}\BinaryInfC{$\phi\wedge\psi$}\noLine\UnaryInfC{} \DisplayProof &
\AxiomC{}\noLine\UnaryInfC{$D$}\noLine\UnaryInfC{$\phi\wedge\psi$} \RightLabel{\ce}\UnaryInfC{$\phi$}\noLine\UnaryInfC{}  \DisplayProof\quad\AxiomC{}\noLine\UnaryInfC{$D$}\noLine\UnaryInfC{$\phi\wedge\psi$} \RightLabel{\ce}\UnaryInfC{$\psi$}\noLine\UnaryInfC{}  \DisplayProof\\\hline
\end{tabular}
\end{center}
\begin{center}
\setlength{\tabcolsep}{6pt}
\renewcommand{\arraystretch}{1.8}
\setlength{\extrarowheight}{1pt}
\begin{tabular}{|C{0.465\linewidth}C{0.465\linewidth}|}
\hline
Boolean disjunction introduction&Boolean disjunction elimination\\
\AxiomC{$D$}\noLine\UnaryInfC{$\phi$} \RightLabel{\bori}\UnaryInfC{$\phi\bor\psi$}  \DisplayProof\quad
\AxiomC{$D$}\noLine\UnaryInfC{$\psi$} \RightLabel{\bori}\UnaryInfC{$\phi\bor\psi$}  \DisplayProof&
\AxiomC{$D_0$}\noLine\UnaryInfC{$\phi\bor\psi$} \AxiomC{}\noLine\UnaryInfC{$[\phi]$}\noLine\UnaryInfC{$D_1$}\noLine\UnaryInfC{$\chi$} \AxiomC{}\noLine\UnaryInfC{$[\psi]$}\noLine\UnaryInfC{$D_2$}\noLine\UnaryInfC{$\chi$} \RightLabel{\bore}\TrinaryInfC{$\chi$}\noLine\UnaryInfC{} \DisplayProof\\\hline
Tensor disjunction weak introduction&Tensor disjunction weakening\\
\AxiomC{}\noLine\UnaryInfC{$D$}\noLine\UnaryInfC{$\phi$}\RightLabel{($\ast$)~\sorwi} \UnaryInfC{$\phi\sor\psi$}  \DisplayProof
&\multirow{2}{*}{\AxiomC{$D$}\noLine\UnaryInfC{$\phi$} 
 \RightLabel{\sorwk}\UnaryInfC{$\phi\sor\phi$} \DisplayProof}\\
{\small($\ast$) whenever $\psi$ dose not contain \nem}&\\\hline
Tensor disjunction weak elimination&Tensor disjunction weak substitution\\
\AxiomC{$D_0$}\noLine\UnaryInfC{$\phi\sor\psi$} \AxiomC{$[\phi]$}\noLine\UnaryInfC{$D_1$}\noLine\UnaryInfC{$\alpha$} \AxiomC{}\noLine\UnaryInfC{$[\psi]$}\noLine\UnaryInfC{$D_2$}\noLine\UnaryInfC{$\alpha$} \RightLabel{$(\ast)$~~\sorwe}\TrinaryInfC{$\alpha$}\DisplayProof 
&\AxiomC{$D_1$}\noLine\UnaryInfC{$\phi\sor\psi$} \AxiomC{}\noLine\UnaryInfC{$[\psi]$}\noLine\UnaryInfC{$D_2$}\noLine\UnaryInfC{$\chi$} \RightLabel{($\ast$)~\sorws}\BinaryInfC{$\phi\sor\chi$} \DisplayProof\\
{\small($\ast$) whenever $\alpha$ is a classical formula and the undischarged assumptions in the derivations $D_1$ and $D_2$ contain classical formulas only}
&{\small($\ast$) whenever the undischarged assumptions in the derivation $D_2$ contain classical formulas only }\\\hline
\multicolumn{2}{|c|}{Commutative and associative laws for tensor disjunction}\\
\AxiomC{}\noLine\UnaryInfC{$D$}\noLine\UnaryInfC{$\phi\sor\psi$} \RightLabel{$\com\sor$}\UnaryInfC{$\psi\sor\phi$}\noLine\UnaryInfC{} \DisplayProof
&
\AxiomC{}\noLine\UnaryInfC{$D$}\noLine\UnaryInfC{$\phi\sor(\psi\sor \chi)$} \RightLabel{$\ass\sor$}\UnaryInfC{$(\phi\sor\psi)\sor\chi$}\noLine\UnaryInfC{} \DisplayProof\\\hline
Contradiction introduction&Contradiction elimination\\
\AxiomC{}\noLine\UnaryInfC{$D$}\noLine\UnaryInfC{$p_i\wedge\neg p_i$}\RightLabel{\boti} \UnaryInfC{$\bot$}\noLine\UnaryInfC{} \DisplayProof
&\AxiomC{}\noLine\UnaryInfC{$D$}\noLine\UnaryInfC{$\phi\sor \bot$} \RightLabel{\bote}\UnaryInfC{$\phi$}\noLine\UnaryInfC{} \DisplayProof\\\hline
%
Strong \emph{ex falso}&Strong contradiction contraction\\
\AxiomC{}\noLine\UnaryInfC{$D$}\noLine\UnaryInfC{$\bot\wedge\nem$} \RightLabel{\sexfalso}\UnaryInfC{$\phi$}\noLine\UnaryInfC{} \DisplayProof
&\AxiomC{}\noLine\UnaryInfC{$D$}\noLine\UnaryInfC{$\phi\sor(\bot\wedge\nem)$} \RightLabel{\sctrc}\UnaryInfC{$\bot\wedge\nem$}\noLine\UnaryInfC{} \DisplayProof\\\hline
\end{tabular}
\end{center}
\begin{center}
\setlength{\tabcolsep}{6pt}
\renewcommand{\arraystretch}{1.8}
\setlength{\extrarowheight}{1pt}
\begin{tabular}{|C{0.465\linewidth}C{0.465\linewidth}|}
\hline
\multicolumn{2}{|c|}{Strong contradiction introduction}\\
\multicolumn{2}{|c|}{\def\ScoreOverhang{0.5pt}
 \def\defaultHypSeparation{\hskip .1in}
\AxiomC{}\noLine\UnaryInfC{$D$}\noLine\UnaryInfC{$\displaystyle\Big(\bigsor_{s\in X}(p_{i_1}^{s(i_1)}\wedge\dots\wedge p_{i_{n}}^{s(i_n)}\wedge\nem)\Big)\wedge\Big(\bigsor_{s\in Y}(p_{i_1}^{s(i_1)}\wedge\dots\wedge p_{i_{n}}^{s(i_n)}\wedge\nem)\Big)$} \RightLabel{($\ast$)~\sctri}\UnaryInfC{$\bot\wedge\nem$}
 \DisplayProof}\\
\multicolumn{2}{|c|}{\small($\ast$) whenever $X$ and $Y$ are distinct teams on $\{i_1,\dots,i_n\}$}\\\hline
\multicolumn{2}{|c|}{Distributive laws}\\
\AxiomC{$D$}\noLine\UnaryInfC{$\phi\otimes(\psi\bor\chi)$} \RightLabel{$\dstr\sor\bor$}\UnaryInfC{$(\phi\otimes \psi)\bor(\phi\otimes \chi)$}  \DisplayProof
&\AxiomC{$D$}\noLine\UnaryInfC{$\displaystyle\nem\wedge\bigsor_{i\in I}\phi_i$} \RightLabel{\dstr$\nem\wedge\sor$}\UnaryInfC{$\displaystyle\bigbor_{\emptyset\neq J\subseteq I}\,\bigsor_{i\in J}(\nem\wedge\phi_i)$}\noLine\UnaryInfC{} \DisplayProof\\\hline
\end{tabular}
\end{center}
\end{definition}

The above deduction system consists of two axioms and fifteen sets of rules.  The atomic excluded middle axiom \exclmid is \emph{not} an axiom schema. Especially, as discussed in the previous section, the substitution instances of $\neg p_i$ are not necessarily well-formed formulas of \PT. The conjunction $\wedge$ and the  Boolean disjunction $\vee$ have the usual introduction and elimination rules. The usual commutative law, associative law and distributive laws for the two connectives can be derived easily in the system. Over classical formulas the tensor disjunction admits the usual Introduction Rule and Elimination Rule. Over non-classical formulas the rules for the tensor disjunction $\otimes$  are more subtle. The Introduction Rule \sorwi  is not in general sound in case $\psi$ does not have the empty team property. For instance, we have  $\bot\not\models\bot\sor\nem$. But the weakening rule \sorwk is sound and we will apply this rule later on in our completeness proof.
The elimination rule \sorwe is not in general sound in case $\alpha$ is non-classical. For instance, we have $p_i\sor \neg p_i\not\models p_i\vee\neg p_i$, even if $p_i\models  p_i\vee\neg p_i$ and $\neg p_i\models  p_i\vee\neg p_i$. The rules \sorwe and \sorws have some side conditions concerning the undischarged assumptions in the sub-derivations. To see why these restrictions are necessary, note, for example, that for the rule \sorwe we have
\[(\nem\wedge p)\sor(\nem\wedge\neg p),p\sor\neg p\not\models \bot,\]
\[\text{ even if }(\nem\wedge p)\sor(\nem\wedge\neg p),p\models \bot\text{ and }(\nem\wedge p)\sor(\nem\wedge\neg p),\neg p\models \bot\]
and for the rule \sorws we have
\[\nem\not\models \nem\sor(\nem\wedge\bot),\text{ even if }\nem\models \nem\sor\bot\text{ and }\nem,\bot\models\nem\wedge\bot.\]
In the absence of the usual (Strong) Introduction Rule and Elimination Rule for $\sor$, we added the (weak) substitution, commutative and associative rules for $\sor$ to the system. We also include some of the distributive laws that involve $\otimes$ in our system, and we derive the other sound distributive laws in the next proposition. The usual distributive laws that are not listed in Definition \ref{Natrual_Deduct_CTLe} are not sound in our system, see \cite{VY_PD} for examples of the failure of these laws. The non-emptiness \nem, the weak and strong negation ($\bot$ and $\bot\wedge\nem$) have self-explanatory  rules in our system. 

\begin{proposition}\label{PT_derivable_rules} The following are derivable in the natural deduction system of \PT:
\begin{description}
\item[(i)] Weak ex falso (\wexfalso): If $\phi$ does not contain \nem, then $\bot\vdash\phi$. 


\item[(ii)] The usual commutative law, associative law and distributive law for conjunction and Boolean disjunction.

\item[(iii)] Distributive laws:

\begin{tabular}{l}
(a)~ $\phi\sor(\psi\wedge\chi)\vdash(\phi\sor \psi)\wedge(\phi\sor \chi)$ ($\dstr\sor\wedge$)\\ 
(b)~ $\phi\bor (\psi\sor \chi)\vdash(\phi\bor\psi)\sor(\phi\bor\chi)$ ($\dstr\bor\sor$)\\
(c)~ $(\phi\otimes \psi)\bor(\phi\otimes \chi)\vdash\phi\otimes(\psi\bor\chi)$ ($\dstr\sor\bor\sor$)\\
(d)~ If $\alpha$ is a classical formula, then \vspace{4pt}\\
\quad\begin{tabular}{l}
- $(\alpha\wedge \psi)\sor(\alpha\wedge \chi)\vdash\alpha\wedge(\psi\sor\chi)$  ($\dstrs\wedge\sor\wedge$)\\
- $\alpha\wedge(\psi\sor\chi)\vdash(\alpha\wedge \psi)\sor(\alpha\wedge \chi)$ ($\dstrs\wedge\sor$)
\end{tabular}
\end{tabular}
%
%
\end{description}
\end{proposition}
\begin{proof}
The rules in item (ii) are derived as usual. 
The items (a), (c) and (d) are not hard to derive (or see the proof of Proposition 4.6 in \cite{VY_PD}). We will only give the derivation for the other rules.

For \wexfalso, assuming that $\phi$ does not contain \nem, we have the following derivation
\begin{center}\AxiomC{$\bot$} \RightLabel{\sorwi}\UnaryInfC{$\bot\sor\phi$}\RightLabel{\bote}\UnaryInfC{$\phi$} \DisplayProof\end{center}



For $\dstr\bor\sor$, we have the following derivation
\begin{center}
 \def\defaultHypSeparation{\hskip .1in}

\AxiomC{$\phi\bor (\psi\sor \chi)$}

\AxiomC{[$\phi$]}
\RightLabel{\sorwk}
\UnaryInfC{$\phi\sor\phi$}

\RightLabel{\bori, \sorws}
\UnaryInfC{$(\phi\bor\psi)\sor(\phi\bor\chi)$}

\AxiomC{[$\psi\sor\chi$]}

\RightLabel{\bori, \sorws}
\UnaryInfC{$(\phi\bor\psi)\sor(\phi\bor\chi)$}

\RightLabel{\bore}
 \TrinaryInfC{$(\phi\bor\psi)\sor(\phi\bor\chi)$}\DisplayProof
\end{center}

%
%
%
\end{proof}

Next, we prove the Soundness Theorem for our deduction system.

\begin{theorem}[Soundness Theorem]\label{PDbornem_soundness}
For any set $\Gamma\cup\{\phi\}$ of formulas in the language of \PT, we have 
\(\Gamma\vdash\phi\,\Longrightarrow\,\Gamma\models\phi.\)
\end{theorem}
\begin{proof}
We show that $\Gamma\models\phi$ holds for each derivation  $D=\{\delta_1,\dots,\delta_k\}$ with the conclusion $\phi$ and the hypotheses in $\Gamma$.

If $D$ is a singleton, then $\phi\in \Gamma$ or $\phi=p_i\sor\neg p_i$ or $\phi=\bot\vee\nem$. In the first case, obviously $\{\phi\}\models\phi$. The last two cases are also easy because the two formulas are clearly valid.

The induction steps for the rules 1-4, 6, 9-12, 14 and the distributive rules $\dstr\sor\vee$ and $\dstr\sor\vee\sor$ are easy to verify. We only check the induction step for the other rules.



$\sorwi$: Assume that $D$ is a derivation for $\Pi\vdash\phi$. We show that $\Pi\models\phi\sor\psi$ whenever $\psi$ is does not contain \nem. Suppose $X\models\theta$ for all $\theta\in \Pi$. By the induction hypothesis, we have $\Pi\models\phi$, thus $X\models\phi$ follows. Now, since $\psi$ does not contain \nem, an easy inductive proof shows it has the empty team property, 
i.e., $\emptyset\models\psi$. 
Hence we obtain $\emptyset\cup X\models\phi\sor\psi$ as required.


$\sorwe$: Assume that $D_0$, $D_1$ and $D_2$ are derivations for $\Pi_0\vdash\phi\sor\psi$, $\Pi_1,\phi\vdash\alpha$ and $\Pi_2,\psi\vdash\alpha$, respectively. Assume that  $\alpha$ is a classical formula and $\Pi_1$ and $\Pi_2$ contain classical formulas only. We show that $\Pi_0,\Pi_1,\Pi_2\models\alpha$ follows from the induction hypothesis $\Pi_0\models\phi\sor\psi$,  $\Pi_1,\phi\models\alpha$ and $\Pi_2,\psi\models\alpha$.  Suppose  $X\models\theta$ for all $\theta\in \Pi_0\cup\Pi_1\cup\Pi_2$. Then we have $X\models\phi\sor\psi$, which means that there exist $Y,Z\subseteq X$ such that $X=Y\cup Z$, $Y\models\phi$ and $Z\models\psi$. For all $\theta\in \Pi_1\cup\Pi_2$, since $\theta$ is a classical formula, by \Cref{cpl_prop} we know that $\theta$ has the downward closure property.  It then follows that  $Y\models\theta_1$ for all $\theta_1\in \Pi_1$ and $Z\models\theta_2$ for all $\theta_2\in \Pi_2$.  Thus $Y\models \alpha$ and $Z\models\alpha$.  Since $\alpha$ is a classical formula,  by \Cref{cpl_prop} we know that $\alpha$ has the union closure property, which yields that $X\models\alpha$.

$\sorws$: Assume that  $D_1$ and $D_2$ are derivations for $\Pi_1\vdash\phi\sor\psi$ and $\Pi_2,\psi\vdash\chi$, respectively and $\Pi_2$ contains classical formulas only. We show that $\Pi_1,\Pi_2\models\phi\sor\chi$ follows from the induction hypothesis $\Pi_1\models\phi\sor\psi$ and $\Pi_2,\psi\models\chi$. Suppose $X\models\theta$ for all $\theta\in \Pi_1\cup \Pi_2$. By the induction hypothesis, we have $X\models \phi\sor\psi$, which means there exist $Y,Z\subseteq X$ such that $X=Y\cup Z$, $Y\models\phi$ and $Z\models\psi$.  For all $\theta\in \Pi_2$, since  $\theta$ is a classical formula, by \Cref{cpl_prop} we know that $\theta$ has the downward closure property. Thus $Z\models\theta$. It then follows from the induction hypothesis that $Z\models\chi$. Hence we conclude $Y\cup Z=X\models\phi\sor\chi$.

\sctri: It suffices to show that if $X$ and $Y$ are distinct teams on $\{i_1,\dots,i_n\}$, then $Z\not\models\Theta_X^\ast\wedge\Theta_Y^\ast$ for all teams $Z$. In view of the locality property, we may assume that $Z$ is a team on $\{i_1,\dots,i_n\}$. By Lemma \ref{X_ThetaX_nem}, if $Z\models\Theta^\ast_X\wedge\Theta^\ast_Y$, then $X=Z=Y$, which contradicts the assumption.

\dstrs$\wedge\sor$: It suffices to show that $ \phi\wedge(\psi\sor\chi)\models(\phi\wedge \psi)\sor(\phi\wedge \chi)$, whenever $\phi$ does not contain \nem. Suppose $X\models\phi\wedge(\psi\sor\chi)$. Then $X\models\phi$, $Y\models\psi$ and $Z\models\chi$ for some teams $Y,Z\subseteq X$ with $Y\cup Z=X$. Since $\phi$ does not contain \nem, an easy inductive proof shows that $\phi$ has the downward closure property. 
Thus $Y\models \phi$ and $Z\models\phi$. Hence $Y\models\phi\wedge \psi$  and $Z\models\phi\wedge\chi$, implying $X\models(\phi\wedge \psi)\sor(\phi\wedge \chi)$.

\dstr\nem$\wedge\sor$: It suffices to show that if $X\models \nem\wedge\bigsor_{i\in I}\phi_i$ for some team $X$, then $X\models\bigsor_{i\in J}(\nem\wedge\phi_i)$ for some nonempty $J\subseteq I$. The assumption implies that $X\neq \emptyset$ and there exist teams $X_i\subseteq X$ for each $i\in I$ such that $\bigcup_{i\in I}X_i=X$ and $X_i\models\phi_i$. Let $J\subseteq I$ be the set of indices $i\in I$ such that $X_i\neq\emptyset$. Since $X\neq\emptyset$, we must have that $J\neq\emptyset$. Hence $X_i\models\nem\wedge\phi_i$ for each $i\in J$ and $X\models \bigsor_{i\in J}(\nem\wedge\phi_i)$.
\end{proof}

We now proceed to prove the main result of this section, the Completeness Theorem for our system. Our argument is similar to that of the proof of the Completeness Theorem for propositional logics of dependence in \cite{VY_PD}. The reader may compare what follows with Section 4.2 in  \cite{VY_PD}. To begin with, below we present a crucial lemma that is very similar to Lemma 3.2 in \cite{VY_PD}.

\begin{lemma}\label{comp_CLTe_main_lm}
For any finite nonempty collections of teams $\{X_f\mid f\in F\}$, $\{Y_g\mid g\in G\}$ on some fixed domain, the following are equivalent: 
\begin{description}
\item[(a)] $\displaystyle\bigbor_{f\in F}\Theta_{X_f}^\ast\models\bigbor_{g\in G}\Theta_{Y_g}^\ast$;
\item[(b)] for each $f\in F$, we have $X_f= Y_{g_f}$ for some $g\in G$.
\end{description}
\end{lemma}
\begin{proof}
Follows easily from Lemma \ref{X_ThetaX_nem}.
%
\end{proof}

Two formulas $\phi$ and $\psi$ are said to be \emph{provably equivalent} (written $\phi\dashv\vdash\psi$) if both $\phi\vdash\psi$ and $\psi\vdash\phi$ hold.
Another crucial step for the completeness proof is to establish that every formula  is provably equivalent in the deduction system to a formula in the disjunctive normal form shown in \Cref{tb:nf}.  Let us now state this as a lemma.

\begin{lemma}\label{DNF_CTLe}
Every formula $\phi(p_{i_1},\dots,p_{i_{n}})$ in the language of \PT is provably equivalent to a formula in the normal form
\begin{equation}\label{NF_CTLe_eq}
\bigbor_{f\in F}\Theta_{X_f}^\ast,~\text{ where }~\Theta_{X_f}^\ast=\bigsor_{s\in X_f}(p_{i_1}^{s(i_1)}\wedge\dots\wedge p_{i_{n}}^{s(i_{n})}\wedge\nem),
\end{equation}
 $\{X_f\mid f\in F\}$  a  finite set of teams on $N=\{i_1,\dots,i_{n}\}$. 
\end{lemma}

The detailed proof of the above lemma will be postponed. We now give the proof of the  Completeness Theorem for our system.

\begin{theorem}[Completeness Theorem]\label{CTLe_completeness}
For  any  formulas $\phi$ and $\psi$ in the language of \PT, we have $\psi\models\phi\Longrightarrow \psi\vdash\phi$. In particular, $\models\phi\Longrightarrow\,\vdash\phi$.
\end{theorem}
\begin{proof}
Suppose $\psi\models\phi$, where 
$\phi=\phi(p_{i_1},\dots,p_{i_n})$ and $\psi=\psi(p_{i_1},\dots,p_{i_n})$. By  \Cref{DNF_CTLe}, we have 
\begin{equation}\label{ctle_compl_eq1}
\phi\dashv\vdash \bigbor_{f\in F}\Theta^\ast_{X_f}\quad\text{ and }\quad\psi\dashv\vdash \bigbor_{g\in G}\Theta^\ast_{Y_g}
\end{equation}
 for some finite sets $\{X_f\mid f\in F\}$ and $\{Y_g\mid g\in G\}$ of teams on $\{i_1,\dots,i_n\}$. The Soundness Theorem and (\ref{ctle_compl_eq1}) imply that
\begin{equation}\label{ctle_compl_eq2}
\bigbor_{f\in F}\Theta^\ast_{X_f}\models\bigbor_{g\in G}\Theta^\ast_{Y_g}.
\end{equation}

If $F=\emptyset$, then $\phi\dashv\vdash\bot\wedge\nem$. We obtain  $\phi\vdash\psi$ by \sexfalso. If $G=\emptyset$, then $\psi\dashv\vdash\bot\wedge\nem$. In view of (\ref{ctle_compl_eq2}), we must have $F=\emptyset$ as well and  $\phi\dashv\vdash\bot\wedge\nem$. Hence $\phi\vdash \psi$.

If $F,G\neq \emptyset$, then by Lemma \ref{comp_CLTe_main_lm},  for each $f\in F$ we have $X_f=Y_{g_f}$ and $\Theta^\ast_{X_f}=\Theta^\ast_{Y_{g_f}}$ for some $g_f\in G$,  
which implies that \(\Theta^\ast_{X_f}\vdash\bigbor_{g\in G}\Theta^\ast_{Y_g}\) by \bori.
Hence, we obtain
\(
\bigbor_{f\in F}\Theta^\ast_{X_f}\vdash\bigbor_{g\in G}\Theta^\ast_{Y_g}
\) by applying  \bore. Finally, in view of (\ref{ctle_compl_eq1}) we conclude that $\phi\vdash\psi$.
\end{proof}

\begin{theorem}[Strong Completeness Theorem]
For  any set $\Gamma\cup\{\phi\}$ of formulas in the language of \PT, we have 
\(\Gamma\models\phi\,\Longrightarrow\,\Gamma\vdash\phi.\)
\end{theorem}
\begin{proof}
By \Cref{CTLe_completeness} and the Compactness Theorem (\Cref{compactness_PI}).
\end{proof}

We end this section by supplying the proof of \Cref{DNF_CTLe}.

\begin{proof}[Proof of \Cref{DNF_CTLe}]
Note that in the statement of the lemma we have fixed a set $\{p_{i_1},\dots,p_{i_n}\}$ of variables. These variables all occur in the formula (\ref{NF_CTLe_eq}) in normal form, but not necessarily all of them actually occur in the formula $\phi(p_{i_1},\dots,p_{i_n})$. In order to take care of this subtle point we first prove the following claim:

\begin{claim*} 
If $\{i_1,\dots,i_m\}\subset \{j_1,\dots,j_k\}$, then any formula $\psi(p_{i_1},\dots,p_{i_m})$ in the normal form is provably equivalent to a formula $\theta(p_{j_1},\dots,p_{j_k})$ in the normal form.
\end{claim*}

\begin{proofclaim}
Without loss of generality we may assume that $K=\{j_1,\dots,j_k\}=\{i_1,\dots,i_m,i_{m+1},\dots,i_k\}$ and $k>m$. 
By the assumption, we have
\[\psi(p_{i_1},\dots,p_{i_m})= \bigbor_{f\in F}\bigsor_{s\in X_f}(p_{i_1}^{s(i_1)}\wedge\dots\wedge p_{i_{m}}^{s(i_{m})}\wedge\nem),\]
where $\{X_f\mid f\in F\}$ is a finite set of teams on $M=\{i_1,\dots,i_m\}$. Let
\[\theta(p_{i_1},\dots,p_{i_k})=\bigbor_{f\in F}\mathop{\bigvee_{Y\subseteq 2^K}}_{Y\upharpoonright M=X_f}\bigsor_{s\in Y}(p_{i_1}^{s(i_1)}\wedge\dots\dots p_{i_k}^{s(i_k)}\wedge\nem).\]

The following derivation proves $\theta\vdash\psi$:
\[
\begin{array}{rll}
(1) & \displaystyle\bigbor_{f\in F}\mathop{\bigvee_{Y\subseteq 2^K}}_{Y\upharpoonright M=X_f}\bigsor_{s\in Y}(p_{i_1}^{s(i_1)}\wedge\dots\dots p_{i_k}^{s(i_k)}\wedge\nem)\\
(2) & \displaystyle\bigbor_{f\in F}\mathop{\bigvee_{Y\subseteq 2^K}}_{Y\upharpoonright M=X_f}\bigsor_{s\in Y}(p_{i_1}^{s(i_1)}\wedge\dots\dots p_{i_m}^{s(i_m)}\wedge\nem)&(\ce, \sorws)\\
(3) & \displaystyle\bigbor_{f\in F}\bigsor_{s\in X_f}(p_{i_1}^{s(i_1)}\wedge\dots\wedge p_{i_{m}}^{s(i_{m})}\wedge\nem)&(\bore)
\end{array}
\]

Conversely, $\psi\vdash\theta$ is proved by the following derivation:
\[
\begin{array}{rl}
(1) & \displaystyle\bigbor_{f\in F}\bigsor_{s\in X_f}(p_{i_1}^{s(i_1)}\wedge\dots\wedge p_{i_{m}}^{s(i_{m})}\wedge\nem)\\
(2) &\displaystyle (p_{i_{m+1}}\sor\neg p_{i_{m+1}})\wedge\dots\wedge (p_{i_k}\sor\neg p_{i_k})\quad\text{(\exclmid, \ci)}\\
(3) &\displaystyle \bigsor_{t\in 2^{K\setminus M}}(p_{i_{m+1}}^{t(i_{m+1})}\wedge\dots\wedge p_{i_{k}}^{t(i_{k})})\quad((2), \dstrs\wedge\sor)\\
(4) &\displaystyle \Big(\bigbor_{f\in F}\bigsor_{s\in X_f}(p_{i_1}^{s(i_1)}\wedge\dots\wedge p_{i_{m}}^{s(i_{m})}\wedge\nem)\Big)\wedge\Big( \bigsor_{t\in 2^{K\setminus M}}(p_{i_{m+1}}^{t(i_{m+1})}\wedge\dots\wedge p_{i_{k}}^{t(i_{k})})\Big)\\
&~~((1), (3), \ci)\\
(5) & \displaystyle \bigbor_{f\in F}\bigsor_{s\in X_f}\Big(p_{i_1}^{s(i_1)}\wedge\dots\wedge p_{i_{m}}^{s(i_{m})}\wedge\big(\nem\wedge  \bigsor_{t\in 2^{K\setminus M}}(p_{i_{m+1}}^{t(i_{m+1})}\wedge\dots\wedge p_{i_{k}}^{t(i_{k})})\big)\Big)\\
&~~( \dstrs\wedge\sor)\\
(6) &  \displaystyle\bigbor_{f\in F}\bigsor_{s\in X_f}\Big(p_{i_1}^{s(i_1)}\wedge\dots\wedge p_{i_{m}}^{s(i_{m})}\wedge\bigvee_{\emptyset\neq Z\subseteq 2^{K\setminus M}}\,\bigsor_{t\in Z}(p_{i_{m+1}}^{t(i_{m+1})}\wedge\dots\wedge p_{i_{k}}^{t(i_{k})}\wedge \nem)\Big)\\
&~~(\dstr\nem\wedge\sor, \sorws)\\
(7) &  \displaystyle\bigbor_{f\in F}\bigsor_{s\in X_f}\,\bigvee_{\emptyset\neq Z\subseteq 2^{K\setminus M}}\Big(p_{i_1}^{s(i_1)}\wedge\dots\wedge p_{i_{m}}^{s(i_{m})}\wedge\bigsor_{t\in Z}(p_{i_{m+1}}^{t(i_{m+1})}\wedge\dots\wedge p_{i_{k}}^{t(i_{k})}\wedge \nem)\Big)\\
(8) &  \displaystyle\bigbor_{f\in F} \,\bigsor_{s\in X_f}\,\bigvee_{\emptyset\neq Z\subseteq 2^{N\setminus M}}\,\bigsor_{t\in Z}(p_{i_1}^{s(i_1)}\wedge\dots\wedge p_{i_{m}}^{s(i_{m})}\wedge p_{i_{m+1}}^{t(i_{m+1})}\wedge\dots\wedge p_{i_{k}}^{t(i_{k})}\wedge \nem)\\
&~~(\dstrs\wedge\sor)\\
\end{array}
\]
\[
\begin{array}{rl}
(9) & \displaystyle\bigbor_{f\in F}\,\bigvee_{G: X_f\to \mathcal{Z}}\bigsor_{s\in X_f}\,\bigsor_{t\in G(s)}(p_{i_1}^{s(i_1)}\wedge\dots\wedge p_{i_{m}}^{s(i_{m})}\wedge p_{i_{m+1}}^{t(i_{m+1})}\wedge\dots\wedge p_{i_{k}}^{t(i_{k})}\wedge \nem)\\
&~~\text{where }\mathcal{Z}=\{Z\subseteq 2^{K\setminus M}\mid Z\neq\emptyset\}\quad(\dstr\sor\bor)\\
(10) & \displaystyle\bigbor_{f\in F}\mathop{\bigvee_{Y\subseteq 2^K}}_{Y\upharpoonright M=X_f}\bigsor_{s\in Y}(p_{i_1}^{s(i_1)}\wedge\dots\dots p_{i_k}^{s(i_k)}\wedge\nem)\quad(\text{since dom}(X_f)=M)
\end{array}
\]
\end{proofclaim}

We now prove \Cref{DNF_CTLe} by induction on $\phi(p_{i_1},\dots,p_{i_n})$.

Case $\phi(p_{i_1},\dots,p_{i_n})=p_{i_k}$. 
We prove that
\(p_{i_k}\dashv\vdash\bot\bor(p_{i_k}\wedge\nem).\)
For $p_{i_k}\vdash\bot\bor(p_{i_k}\wedge\nem)$, we have the following derivation: \allowdisplaybreaks
\[
\begin{array}{rll}
(1)& p_{i_k}\\
(2)&\bot\bor\nem&\text{(\nemi)}\\
(3)&\big( p_{i_k}\wedge\bot\big)\bor(p_{i_k}\wedge\nem)&\text{((1), (2), \ci, \dstr)}\\
(4)&\bot\bor(p_{i_k}\wedge\nem)&\text{(\ce)}
\end{array}
\]
Conversely, for $\bot\bor(p_{i_k}\wedge\nem)\vdash p_{i_k}$, we have the following derivation
\allowdisplaybreaks
\[
\begin{array}{rll}
(1)&\bot\bor(p_{i_k}\wedge\nem)\\
(2)& p_{i_k}\vee p_{i_k}&\text{(\wexfalso, \ce)}\\
(3)& p_{i_k}&\text{(\bore)}
\end{array}
\]
 By the Claim, the formula $p_{i_k}$ is provably equivalent to a formula $\theta(p_{i_1},\dots,p_{i_n})$ in the normal form.

\vspace{0.5\baselineskip}

Case $\phi(p_{i_1},\dots,p_{i_n})=\neg p_{i_k}$. Similar to the above case.

\vspace{0.5\baselineskip}

%
%
Case $\phi(p_{i_1},\dots,p_{i_n})=\nem$. 
Note that \nem is a formula with no propositional variable, but for the sake of the inductive proof, we need to prove the theorem for \nem viewed as $\nem(p_{i_1},\dots,p_{i_n})$, a formula whose propositional variables are among $p_{i_1},\dots,p_{i_n}$. 
By the claim, it suffices to derive the normal form for \nem when it is viewed as $\nem(p_{i_1})$. We prove that $\nem\dashv\vdash\theta$, where
\begin{equation*}
\theta:= (p_{i_1}\wedge \nem)\vee(\neg p_{i_1}\wedge \nem)\vee \big((p_{i_1}\wedge \nem)\sor (\neg p_{i_1}\wedge \nem)\big).
\end{equation*}

For $\nem\vdash\theta$, we have the following derivation:
\allowdisplaybreaks
\[
\begin{array}{rll}
(1)&\nem\\
(2)&\nem\wedge(p_{i_1}\sor\neg p_{i_1})&\text{(\exclmid, \ci)}\\
(3)&(p_{i_1}\wedge \nem)\vee(\neg p_{i_1}\wedge \nem)\vee \big((p_{i_1}\wedge \nem)\sor (\neg p_{i_1}\wedge \nem)\big)& \text{(\dstr\,$\nem\wedge\sor$)}
\end{array}
\]

For  the other direction $\theta\vdash \nem$, we have the following derivation:
\allowdisplaybreaks
\[
\begin{array}{rll}
(1)&(p_{i_1}\wedge \nem)\vee(\neg p_{i_1}\wedge \nem)\vee \big((p_{i_1}\wedge \nem)\sor (\neg p_{i_1}\wedge \nem)\big)\\
(2)&(p_{i_1}\wedge \nem)\vee(\neg p_{i_1}\wedge \nem)\vee \big((p_{i_1}\sor \neg p_{i_1})\wedge \nem\big)& \text{(\dstrs$\wedge\sor\wedge$)}\\
(3)&\nem\vee\nem\vee \nem& \text{(\ce)}\\
(4) &\nem&\text{(\bore)}
\end{array}
\]

\vspace{0.5\baselineskip}

Case $\phi(p_{i_1},\dots,p_{i_n})=\bot$. Trivially $\bot\dashv\vdash\Theta^\ast_\emptyset=\bot$.

\vspace{0.5\baselineskip}

\vspace{0.5\baselineskip}

 Case $\phi(p_{i_1},\dots,p_{i_n})=\psi(p_{i_1},\dots,p_{i_n})\bor\chi(p_{i_1},\dots,p_{i_n})$. By the induction hypothesis, we have
  \begin{equation}\label{pdbornem_nf_proof_IH}
  \psi\dashv\vdash\bigbor_{f\in F}\Theta^\ast_{X_f}\text{ and }\chi\dashv\vdash\bigbor_{g\in G}\Theta^\ast_{X_g},
  \end{equation}
where  each $X_{f},X_g\subseteq 2^{N}$. Then it follows from the rules \bore and \bori that 
\[\psi\bor\chi\dashv\vdash\bigbor_{f\in F}\Theta_{X_f}^\ast\vee\bigbor_{g\in G}\Theta_{X_g}^\ast.\]
If $\psi\dashv\vdash \bot\wedge\nem$ (i.e., $F=\emptyset$), then we obtain further by \sexfalso, \bore and \bori that $\psi\bor\chi\dashv\vdash\bigbor_{g\in G}\Theta_{X_g}^\ast$. Similarly for the case $\chi\dashv\vdash \bot\wedge\nem$.


\vspace{0.5\baselineskip}

Case $\phi(p_{i_1},\dots,p_{i_n})=\psi(p_{i_1},\dots,p_{i_n})\sor\chi(p_{i_1},\dots,p_{i_n})$. By the induction hypothesis, we have (\ref{pdbornem_nf_proof_IH}). If $\psi\dashv\vdash \bot\wedge\nem$ (i.e., $F=\emptyset$), then we derive $\psi\sor\chi\dashv\vdash\bot\wedge\nem=\bigbor\emptyset$ by (\sctrc) and (\sexfalso). Similarly for the case $\chi\dashv\vdash \bot\wedge\nem$ (i.e., $G=\emptyset$).


Now, assume $F,G\neq\emptyset$. We show that $\psi\sor\chi\dashv\vdash \theta$, where 
\[\theta:= \bigbor_{f\in F}\bigbor_{g\in G}\Theta^\ast_{X_f\cup X_g}.\]
For the direction $\psi\sor\chi\vdash\theta$, we have the following derivation:
\allowdisplaybreaks
\[
\begin{array}{rll}
(1)&\psi\sor\chi\\
(2)&\displaystyle\Big(\bigbor_{f\in F}\Theta^\ast_{X_f}\Big)\sor\Big(\bigbor_{g\in G}\Theta^\ast_{X_g}\Big)\\
(3)&\displaystyle\bigbor_{f\in F}\bigbor_{g\in G}\left(\Theta^\ast_{X_f}\sor\Theta^\ast_{X_g}\right)&\text{(\dstr$\sor\bor$)}\\
(4)&\displaystyle\bigbor_{f\in F}\bigbor_{g\in G}\Theta^\ast_{X_f\cup X_g}&(\sorwe, \sorws)
\end{array}
\]
The other direction $\theta\vdash\psi\sor\chi$ is proved similarly using \sorwk and $\dstr\sor\bor\sor$.

\vspace{0.5\baselineskip}

Case $\phi(p_{i_1},\dots,p_{i_n})=\psi(p_{i_1},\dots,p_{i_n})\wedge\chi(p_{i_1},\dots,p_{i_n})$. By the induction hypothesis, we have (\ref{pdbornem_nf_proof_IH}).  If $\psi\dashv\vdash \bot\wedge\nem$ (i.e., $F=\emptyset$), then we derive $\psi\wedge\chi\dashv\vdash \bot\wedge\nem=\bigbor\emptyset$ by (\ce) and (\sexfalso). Similarly for the case $\chi\dashv\vdash \bot\wedge\nem$ (i.e., $G=\emptyset$). 

Now, assume $F,G\neq\emptyset$. We show that $\psi\wedge\chi\dashv\vdash \theta$, where
\[\theta:=  \bigbor_{h\in H}\Theta^\ast_{X_h}\text{ and }\{X_f\mid f\in F\}\cap\{X_g\mid g\in G\}=\{X_h\mid h\in H\}.\]
For $\psi\wedge\chi\vdash\theta$, we have the following derivation:
\allowdisplaybreaks
\[\begin{array}{rll}
(1)&\psi\wedge\chi\\
(2)&\displaystyle\Big(\bigbor_{f\in F}\Theta^\ast_{X_f}\Big)\wedge\Big(\bigbor_{g\in G}\Theta^\ast_{X_g}\Big)\\
(3)&\displaystyle\bigbor_{f\in F}\bigbor_{g\in G}\left(\Theta^\ast_{X_f}\wedge\Theta^\ast_{X_g}\right)\\
(4)&\displaystyle\Big(\mathop{\bigbor_{(f,g)\in F\times G}}_{X_f\neq X_g}\left(\Theta^\ast_{X_f}\wedge\Theta^\ast_{X_g}\right)\Big)\bor\Big(\mathop{\bigbor_{(f,g)\in F\times G}}_{X_f=X_g}\left(\Theta^\ast_{X_f}\wedge\Theta^\ast_{X_g}\right)\Big)\\
\end{array}
\]
\[\begin{array}{rll}
(5)&\displaystyle\big(\bot\wedge\nem\big)\bor\mathop{\bigbor_{(f,g)\in F\times G}}_{X_f=X_g}\left(\Theta^\ast_{X_f}\wedge\Theta^\ast_{X_g}\right)&(\sctri)\quad\quad\quad\quad\quad\quad\quad\quad\\
(6)&\displaystyle\mathop{\bigbor_{(f,g)\in F\times G}}_{X_f=X_g}\left(\Theta^\ast_{X_f}\wedge\Theta^\ast_{X_g}\right)&(\sctre)\\
(7)&\displaystyle\bigbor_{h\in H}\Theta^\ast_{X_h}&(\ce, \bore)
\end{array}
\]
\normalsize
For the other direction $\theta\vdash\psi\wedge\chi$, we have the following derivation:
\allowdisplaybreaks
\[
\begin{array}{rll}
(1) & \displaystyle \bigbor_{h\in H}\Theta^\ast_{X_h}\\
(2)&\displaystyle\Big(\bigbor_{h\in H}\Theta^\ast_{X_h}\Big)\wedge\Big(\bigbor_{h\in H}\Theta^\ast_{X_h}\Big)&(\ci)\\
(3)&\displaystyle\Big(\bigbor_{f\in F}\Theta^\ast_{X_f}\Big)\wedge\Big(\bigbor_{g\in G}\Theta^\ast_{X_g}\Big)&(\text{\bori,  $H\subseteq F,G$})\\
(4)&\displaystyle\psi\wedge\chi
\end{array}
\]
\end{proof}

\subsection{\ECL}\label{sec:ctl}


We will give a complete axiomatization of \ECL in the style of natural deduction. For this end, let us review our  proof above of the Completeness Theorem for \PT.  In a crucial step of the proof we transformed a formula into its disjunctive normal form $\bigvee_{f\in F}\Theta^\ast_{X_f}$. Each disjunct $\Theta^\ast_{X_f}$ is a formula in the language of  \ECL, but \ECL is a fragment of \PT that does not have the Boolean disjunction $\vee$ in the language, so we seem to be in trouble. Our trick is that we view the set $\{\Theta^\ast_{X_f}\mid f\in F\}$ of formulas, rather than the disjunction of this set, as a \emph{weak} normal form for formulas in the language of \ECL.  On the basis of this plan, we can  axiomatize \ECL and prove the completeness theorem.



We will define a natural deduction system of \ECL in which  every formula $\phi$ is essentially provably equivalent to its disjunctive normal form $\bigvee_{f\in F}\Theta^\ast_{X_f}$. In particular, we will be able to essentially derive the provable equivalence between the non-emptiness \nem and its disjunctive normal form $\bigvee_{\emptyset\neq Y\subseteq 2^N}\Theta^\ast_Y$. The behavior of the usual Introduction Rule and  Elimination Rule  of the Boolean disjunction $\vee$ (\bori and \bore)  will be simulated by two Strong Elimination Rules (\sen and \see) that do not involve $\vee$.

To define the Strong Elimination Rules, we will need to specify a particular occurrence of a subformula inside a formula. For this purpose, we identify a formula in the language of \ECL with a finite string of symbols. A propositional variable $p_i$ and the non-emptiness \nem are symbols and the other symbols are $\wedge,\sor,\neg$. Starting from the leftmost symbol, we number each symbol in a formula with a positive integer, as in the following example:
\begin{center}
\begin{tabular}{cccccccc}
$\nem$&$\sor$&$($&$\neg$&$p_1$&$\wedge$&$\nem$&$)$\vspace{4pt}\\
1&2&3&4&5&6&7&8
\end{tabular}
\end{center}
%
%
If the $m$th symbol of a formula $\phi$ starts a string $\psi$ which is a subformula of $\phi$, we denote the subformula  $[\psi,m]_\phi$, or simply $[\psi,m]$. We will sometimes refer to an occurrence of a formula $\chi$ inside a subformula $\psi$ of $\phi$. In this case we will  use the same counting for the subformula $\psi$, rather than restart the counting from $1$. We write $\phi(\beta/[\alpha,m])$ for the formula obtained from $\phi$ by replacing the occurrence of the subformula $[\alpha,m]$ by $\beta$. For example, for the  formula $\phi=\nem\sor(\neg p_1\wedge\nem)$,  the second occurrence of the non-emptiness \nem is denoted by  $[\nem,7]$, and the same notation also designates the occurrence of $\nem$ inside the subformula $\neg p_1\sor\nem$. The notation $\phi(\psi/[\nem,7])$ designates the formula $\nem\sor(\neg p_1\wedge \psi)$.

Below we present the natural deduction system of \ECL.

\begin{definition}[A natural deduction system of \ECL]\label{Natrual_Deduct_ctl}\ 

\begin{center}
\setlength{\tabcolsep}{6pt}
\renewcommand{\arraystretch}{1.8}
\setlength{\extrarowheight}{1pt}
\begin{tabular}{|C{0.95\linewidth}|}
\multicolumn{1}{c}{\textbf{AXIOM}}\\\hline\hline
Atomic excluded middle\\
\AxiomC{}\noLine\UnaryInfC{} \RightLabel{\exclmid}\UnaryInfC{$p_i\sor\neg p_i$}\noLine\UnaryInfC{}\DisplayProof\\\hline
\end{tabular}
\end{center}


\begin{center}
\setlength{\tabcolsep}{6pt}
\renewcommand{\arraystretch}{1.8}
\setlength{\extrarowheight}{2pt}
\begin{tabular}{|C{0.95\linewidth}|}
 \multicolumn{1}{c}{\textbf{RULES}}\\\hline\hline
All of the rules in \Cref{Natrual_Deduct_CTLe} that do not involve Boolean disjunction, i.e., the rules \ci, \ce, \sorwi, \sorwk, \sorwe, \sorws,  \com\sor, \ass\sor, \boti, \bote,  \sexfalso, \sctri, \sctrc, \dstrs$\wedge\sor$.\\\hline
Strong elimination rules\\
\def\ScoreOverhang{0.5pt}
\def\defaultHypSeparation{\hskip .1in}
\AxiomC{$D_0$}\noLine\UnaryInfC{$\phi$} \AxiomC{}\noLine\UnaryInfC{[$\phi(\Theta^\ast_{Y_1}/[\nem,m])$]}\noLine\UnaryInfC{$D_1$} \noLine\UnaryInfC{$\theta$} \AxiomC{\null}\noLine\UnaryInfC{\null}\noLine\UnaryInfC{$\ldots$}\AxiomC{}\noLine\UnaryInfC{[$\phi(\Theta^\ast_{Y_k}/[\nem,m])$]}\noLine\UnaryInfC{$D_k$}\noLine\UnaryInfC{$\theta$}\RightLabel{\sen}\QuaternaryInfC{$\theta$} \DisplayProof\\
{\small where $\{Y_1,\dots,Y_k\}$ is the set of all nonempty teams on a set $N$ of indices}\\
\AxiomC{$D_0$}\noLine\UnaryInfC{$\phi$} \AxiomC{}\noLine\UnaryInfC{[$\phi(\psi\wedge\bot/[\psi,m])$]}\noLine\UnaryInfC{$D_1$} \noLine\UnaryInfC{$\theta$} \AxiomC{}\noLine\UnaryInfC{}\noLine\UnaryInfC{[$\phi(\psi\wedge\nem/[\psi,m])$]}\noLine\UnaryInfC{$D_2$}\noLine\UnaryInfC{$\theta$} \RightLabel{\see}\TrinaryInfC{$\theta$}\noLine\UnaryInfC{} \DisplayProof\\\hline
\end{tabular}
\end{center}
\end{definition}


All rules that do not involve the Boolean disjunction $\vee$ in the natural deduction system of \PT (\Cref{Natrual_Deduct_CTLe}) are included in the above system. Thus all clauses in \Cref{PT_derivable_rules} that do not involve $\vee$ are also derivable in the above system. 

Let us ponder why we define the Strong Elimination Rules the way they are in our system in the absence of the Boolean disjunction $\vee$. The idea of the elimination rules for the conjunction is simply that if we have inferred $\phi\wedge\psi$, we can infer both $\phi$ and $\psi$. The elimination rule for disjunction in classical logic is that if we have $\phi\vee\psi$ and we can derive $\theta$ separately from both $\phi$ and $\psi$ then we have $\theta$. In both cases the elimination rule builds into the syntax of the proof the semantics of the logical operation. This is roughly the general idea of natural deduction, due to Gentzen. We followed the same line of thinking when we introduce our rules for Boolean disjunction $\vee$ in the system for \PT. Now, we moved to the weak logic \ECL which does not have Boolean disjunction in the language. But still, the non-emptiness \nem is semantically equivalent to the formula $\bigvee_{i=1}^k\Theta^\ast_{Y_i}$ (in the language of \PT), where $\{Y_1,\dots,Y_k\}$ is the set of all nonempty teams on a set $N$ of indices. To derive a formula $\theta$ from \nem given some assumptions, in \PT in the presence of the Boolean disjunction we could build up the following derivation:
\begin{center}\AxiomC{$D_0$}\noLine\UnaryInfC{$\nem$}\UnaryInfC{$\bigvee_{i=1}^k\Theta^\ast_{Y_i}$} \AxiomC{[$\Theta^\ast_{Y_1}$]}\noLine\UnaryInfC{$D_1$} \noLine\UnaryInfC{$\theta$} \AxiomC{\null}\noLine\UnaryInfC{\null}\noLine\UnaryInfC{$\ldots$}\AxiomC{[$\Theta^\ast_{Y_k}$]}\noLine\UnaryInfC{$D_k$}\noLine\UnaryInfC{$\theta$}\QuaternaryInfC{$\theta$} \DisplayProof \end{center}
Evidently such a derivation can be simulated in our deduction system of \ECL using the rule \sen. More generally, a formula $\phi$ whose $n$th symbol is \nem is semantically equivalent to the formula $\phi(\bigvee_{i=1}^k\Theta^\ast_{Y_i}/[\nem,m])$ (in the language of \PT), which, as $\vee$ distributes over all connectives, is semantically equivalent to $\bigvee_{i=1}^k\phi(\Theta^\ast_{Y_i}/[\nem,m])$. In the same way, to derive a formula $\theta$ from $\phi$ given some assumptions, it suffices to derive $\theta$ from each $\phi(\Theta^\ast_{Y_i}/[\nem,m])$. This is exactly what the Strong Elimination Rule \sen characterizes. Analogously, the rule \see characterizes the equivalence between a formula $\phi$ and $\phi(\psi\wedge(\bot\vee\nem)/[\psi,m])$.




We now prove the Soundness Theorem for the deduction system of \ECL.

\begin{theorem}[Soundness Theorem]\label{ctl_soundness}
For any set $\Gamma\cup\{\phi\}$ of formulas in the language of \ECL, we have 
\(\Gamma\vdash\phi\,\Longrightarrow\,\Gamma\models\phi.\)
\end{theorem}
\begin{proof}
We show that for each derivation $D$ with the conclusion $\phi$ and the hypotheses in $\Gamma$ we have  $\Gamma\models\phi$.
We only verify the cases where the Strong Elimination Rules are applied. The other cases follow from the Soundness Theorem for \PT.

\sen: Put $\phi^\ast_i=\phi(\Theta^\ast_{Y_i}/[\nem,m])$ for each $i\in\{1,\dots,k\}$. Assume that $D_0,D_1,\dots,D_k$ are derivations for $\Pi_0\vdash\phi$,\, $\Pi_1,\phi^\ast_1\vdash\theta$, ..., $\Pi_k,\phi^\ast_k\vdash\theta$, respectively. We show that $\Pi_0,\Pi_1,\dots,\Pi_k\models\theta$ follows from the induction hypothesis $\Pi_0\models\phi$,\, $\Pi_1,\phi^\ast_1\models\theta$, ..., $\Pi_k,\phi^\ast_k\models\theta$. This reduces to showing that $\phi\models\phi^\ast_1\vee\dots\vee\phi^\ast_k$ by induction on the subformulas $\psi$ of $\phi$.

Case $\psi=\nem$. By the locality property, $X\models\nem\iff X\upharpoonright N \models\nem$ for any team $X$. Now, since $\{Y_1,\dots,Y_k\}$ is the set of all nonempty teams on $N$, we have $X\upharpoonright N\models\nem\iff X\upharpoonright N\models \Theta^\ast_{Y_i}$ for some $i\in\{1,\dots,k\}$. Hence $\nem\models\Theta^\ast_{Y_1}\vee\dots\vee\Theta^\ast_{Y_k}$.

If $\psi$ is $p_j$ or $\neg p_j$, then $\psi^\ast_i=\psi$ for each $i\in\{1,\dots,k\}$ and $\psi\models\psi^\ast_1\vee\dots\vee\psi^\ast_k$ holds trivially.

If $\psi=\theta\sor\chi$ and without loss of generality we assume that the occurrence of the $\nem$ is in the subformula $\theta$. Then by the induction hypothesis we have  $\theta\models \theta^\ast_1\vee\dots\vee\theta^\ast_k$. Thus $\theta\sor\chi\models (\theta^\ast_1\vee\dots\vee\theta^\ast_k)\sor\chi\models (\theta^\ast_1\sor\chi)\vee\dots\vee(\theta^\ast_k\sor\chi)$.

The case $\psi=\theta\wedge\chi$ is proved similarly.

\see: Put $\phi^\ast_+=\phi(\psi\wedge\nem/[\psi,m])$ and $\phi^\ast_-=\phi(\psi\wedge\bot/[\psi,m])$. Assume that $D_0,D_1$ and $D_2$ are derivations for $\Pi_0\vdash\phi$,\, $\Pi_1,\phi^\ast_+\vdash\theta$ and $\Pi_2,\phi^\ast_-\vdash\theta$, respectively. We show that $\Pi_0,\Pi_1,\Pi_2\models\theta$ follows from the induction hypothesis $\Pi_0\models\phi$,\, $\Pi_1,\phi^\ast_+\models\theta$ and $\Pi_2,\phi^\ast_-\models\theta$. This is reduced to showing that $\phi\models\phi^\ast_+\vee\phi^\ast_-$ by induction on the subformulas $\delta$ of $\phi$. 

If $\delta$ is an atom and $\delta\neq [\psi,m]$, then $\delta^\ast_+=\delta=\delta^\ast_-$ and $\delta\models\delta^\ast_+\vee\delta^\ast_-$ holds trivially.

If $\delta=[\psi,m]$, then  $\delta^\ast_+=\delta\wedge\nem$ and $\delta^\ast_-=\delta\wedge\bot$. Since $\models\nem\vee\bot$, we have $\delta\models\delta\wedge(\nem\vee\bot)\models(\delta\wedge\nem)\vee(\delta\wedge\bot)$.

The induction steps are proved analogously to the \sen case.
\end{proof}

In the remainder of this section we prove the Completeness Theorem for our system. This proof is similar to the completeness proof for \PD that we gave in \cite{VY_PD}. The reader may compare this section with Section 4.3 in \cite{VY_PD}. When proving the Completeness Theorem for the deduction system of \PT, we transformed  a formula into its disjunctive normal form. Here in \ECL we follow essentially the same idea. But in the absence of the Boolean disjunction we will not be able to express the relevant disjunctive normal form in the logic. Instead, we work with the weak normal form (i.e., the set of all disjuncts of a disjunctive normal form) and the behavior of the (strong) normal form can be simulated by using the Strong Elimination Rules \sen and \see. The disjuncts of the disjunctive normal of a formula (or elements in the weak normal form) can be obtained from what we call \emph{strong realizations}. Our strong realizations are analogues of the ``resolutions" in \cite{Ciardelli_dissertation}, and they are more complex than the ``(weak) realizations" we defined in \cite{VY_PD}. These strong realizations will play a crucial role in our argument. Let us now define this notion formally.

 Let $\alpha\in\{p_i,\neg p_i,\nem,\bot\}$ be an atom and $Y\subseteq 2^N$ a team on a set $N$ of indices such that $Y\models\alpha$. A \emph{strong realization} $\alpha^\ast_Y$ of $\alpha$ over $Y$ is defined as 
 \[\alpha^\ast_Y:=\Theta^\ast_Y.\]
Let 
\(o=\langle [\alpha_1,m_1],\,\dots,\,[\alpha_c,m_c]\rangle\)
be a sequence of some of the occurrences of atoms in a formula $\phi$ in the language of \ECL. A \emph{strongly realizing sequence of $\phi$ over $o$} is a sequence $\Omega=\langle Y_1,\dots,Y_{c}\rangle$ such that $Y_i\models\alpha_i$ for each $i\in\{1,\dots,c\}$. We call the formula $\phi_\Omega^\ast$ defined 
as follows a \emph{strong realization of $\phi$ over $o$}: 
\[\phi_{\langle Y_1,\dots,Y_{c}\rangle}^\ast:=\phi((\alpha_1)_{Y_1}^\ast/[\alpha_1,m_1],\dots,(\alpha_c)_{Y_c}^\ast/[\alpha_c,m_c]).\]
Let $O$ be the sequence of all occurrences of all atoms in $\phi$. A strongly realizing sequence of $\phi$ over $O$ is called a \emph{maximal strongly realizing sequence}. A strong realization $\phi^\ast_\Omega$ over $O$ is called a \emph{strong realization} of $\phi$. 

For example, consider the formula $\phi=\nem\sor(\neg p_1\wedge\nem)$. Let $Y_1$ and $Y_2$ be two nonempty teams and $X=\{\{(1,0)\}\}$ a team on $\{1\}$. Over $o=\langle[\nem,7]\rangle$ the sequence $\langle Y_1\rangle$ is a strongly realizing sequence of $\phi$ and the formula $\nem\sor(\neg p_1\wedge\nem^\ast_{Y_1})$  is a strong realization of $\phi$. Both $\nem^\ast_{Y_1}\sor((\neg p_1)^\ast_X\wedge\nem^\ast_{Y_2})$ and $\nem^\ast_{Y_2}\sor((\neg p_1)^\ast_X\wedge\nem^\ast_{Y_1})$ are strong realizations of $\phi$. Note that a formula always has at least one atom, so its maximal strongly realizing sequence is always a nonempty sequence.

In the next lemma we prove that every formula is semantically equivalent to the Boolean disjunction of all of its strong realizations over an arbitrary sequence of some occurrences of atoms, particularly of all its maximal strongly realizing sequences.

\begin{lemma}\label{WNF_ctl_semantic}
Let $\phi$ be a formula in the language of \ECL and $\Lambda$ the set of its strongly realizing sequences over a sequence $o$. 
Then 
\(\displaystyle\phi\equiv \bigbor_{\Omega\in \Lambda}\phi^\ast_\Omega.\)
\end{lemma}

\begin{proof}
We prove the lemma by induction on the subformulas $\psi$ of $\phi$. Let $N$ be the set of indices of propositional variables occurring in $\phi$. 

Base case: $\psi$ is an atom. If the occurrence of $\psi$ is not listed in $o$, then $\psi^\ast_\Omega=\psi$ for all $\Omega\in\Lambda$ and  $\psi\equiv \bigvee_{\Omega\in\Lambda}\psi^\ast_\Omega$ holds trivially. Otherwise, the occurrence $\psi=[\psi,m_i]$ is in $o$ and the set $\mathcal{Y}=\{Y_i\mid \langle Y_1,\dots,Y_c\rangle\in \Lambda\}$ consists of all teams on $N$ that satisfy $\psi$. For any team $X$ on $N$, by \Cref{comp_CLTe_main_lm} we have $X\models\psi\iff X\in\mathcal{Y}\iff X\models \bigvee_{Y\in\mathcal{Y}}\Theta^\ast_{Y}\iff X\models\bigvee_{\Omega\in\Lambda}\psi^\ast_\Omega$.


The induction case $\psi=\theta\sor\chi$ follows from the induction hypothesis and the fact that 
\([\,A\models A'\text{ and }B\models B'\,] \Longrightarrow A\sor B\models A'\sor B'\)
and that
$A\sor(B\bor C)\models (A\sor B)\bor(A\sor C)$ for all formulas $A,B,C$. Analogously for the case $\psi=\theta\wedge\chi$.
\end{proof}

We will show that in our system one derives essentially the equivalence between a formula $\phi$ and the Boolean disjunction $\bigbor_{\Omega\in \Lambda}\phi^\ast_\Omega$ of its strong realizations over some sequence of occurrences of atoms. We first prove the direction that $\phi$ follows from $\bigbor_{\Omega\in \Lambda}\phi^\ast_\Omega$, which is simulated in our system by the derivation that each Boolean disjunct $\phi^\ast_\Omega$ implies $\phi$.

\begin{lemma}\label{se_2_form} 
If $\Omega$ is a strongly realizing sequence of a formula $\phi$ in the language of \ECL over a sequence of some occurrences of  atoms in $\phi$, then $\phi^\ast_\Omega\vdash\phi$.
\end{lemma}
\begin{proof}
We derive the lemma by induction on the subformulas $\psi$ of $\phi$.  Let $N=\{i_1,\dots,i_n\}$ be the set of indices of propositional variables occurring in $\phi$. 

The induction step is left to the reader. We only check the basic case when $\psi$ is an atom.  If the occurrence of $\psi$ is not listed in $o$, then $\psi^\ast_{\Omega}=\psi$ and the statment holds trivially. Now, assume otherwise. Then $\psi^\ast_{\Omega}=\Theta^\ast_X$ and $X$ is a team on  $N$ that satisfies $\psi$. If $\psi=\bot$, then $X=\emptyset$ and $\Theta^\ast_\emptyset=\bot$. Thus $\psi^\ast_{\Omega}\vdash\bot$ holds trivially. If $\psi=p_{i_k}$ and  $X=\emptyset$, then $\Theta^\ast_\emptyset=\bot\vdash p_{i_k}$ follows from  \wexfalso. If $X\neq\emptyset$, then we have $s(i_k)=1$ for all $s\in X$ and we derive $\Theta^\ast_X\vdash p_i$ as follows:
\[
\begin{array}{rll}
  (1)&\displaystyle\bigsor_{s\in X}(p_{i_1}^{s(i_1)}\wedge\dots\wedge p_{i_{k-1}}^{s(i_{k-1})}\wedge p_{i_k}\wedge p_{i_{k+1}}^{s(i_{k+1})}\wedge\dots\wedge p_{i_n}^{s(i_n)}\wedge\nem)\\
 (3)&\displaystyle\bigsor_{s\in X}p_{i_k}&(\ce, \sorws)\\
 (4)& p_{i_k}&(\sorwe)
\end{array}
\]
 The case $\psi=\neg p_{i_k}$ is proved analogously. If $\psi=\nem$, then $X\neq\emptyset$ and $\Theta^\ast_X\vdash \nem$ is derived by a similar argument.
\end{proof}

Next, we turn to prove that the Boolean disjunction $\bigbor_{\Omega\in \Lambda}\phi^\ast_\Omega$ of the strong realizations of some formula $\phi$ over some sequence $o$ of occurrences of atoms follows essentially from $\phi$. Following the idea we used when we defined the Strong Elimination Rules, we simulate the derivation by proving in our system that a formula $\theta$ follows from $\phi$, given that $\theta$ follows from each $\phi^\ast_\Omega$. We prove this  in steps. First of all, if $o$ is a sequence of one occurrence of the non-emptiness \nem, then the statement follows by applying the \sen rule. We now generalize this result and show that the statement holds if $o$ is a sequence of one occurrence of any atom.

\begin{lemma}\label{replacement_lm}
If $\delta\vdash\theta$, then $\phi(\delta/[\psi,m])\vdash\phi(\theta/[\psi,m])$. 
\end{lemma}
\begin{proof}
We prove the lemma by induction on the subformulas $\chi$ of $\phi$. 

If $\chi$ is an atom and $\chi\neq [\psi,m]$, then $\chi(\delta/[\psi,m])=\chi=\chi(\theta/[\psi,m])$ and trivially $\chi(\delta/[\psi,m])\vdash\chi(\theta/[\psi,m])$.

If $\chi=[\psi,m]$, then $\chi(\delta/[\psi,m])=\delta$ and $\chi(\theta/[\psi,m])=\theta$. Thus $\chi(\delta/[\psi,m])\vdash\chi(\theta/[\psi,m])$ follows directly from the assumption.

Suppose $\chi=\chi_0\sor\chi_1$. Without loss of generality we may assume that the occurrence of the formula $\psi$ is in the subformula $\chi_0$. By the induction hypothesis we have $\chi_0(\delta/[\psi,m])\vdash\chi_0(\theta/[\psi,m])$. An application of the rule \sorws yields $\chi_0(\delta/[\psi,m])\sor\chi_1\vdash\chi_0(\theta/[\psi,m])\sor\chi_1$.

The case $\chi=\chi_0\wedge\chi_1$ is proved analogously by applying \ce and \ci.
\end{proof}

\begin{lemma}\label{general_se_single_lm}
Let $[\alpha,m]$ be an occurrence of an atom in a formula $\phi$ and $\mathcal{Y}$ the set of all teams on $N=\{i_1,\dots,i_n\}$ that satisfy $\alpha$.
 For any set $\Gamma\cup\{\theta\}$ of formulas in the language of \ECL, if $\Gamma,\phi^\ast_{Y}\vdash\theta$ for all $Y\in \mathcal{Y}$, then $\Gamma,\phi\vdash\theta$.
\end{lemma}
\begin{proof}
If $\alpha$ is the non-emptiness \nem, then the statement follows from \sen. If $\alpha=\bot$, then $\mathcal{Y}=\{\emptyset\}$, $\Theta^\ast_\emptyset=\bot$ and $\phi^\ast_\emptyset=\phi$. Thus the statement holds trivially. The nontrivial case is when $\alpha$ is $p_{i_k}$ or $\neg p_{i_k}$. We only give the proof for the case $\alpha=p_{i_k}$. The case  $\alpha=\neg p_{i_k}$ is proved similarly.

In view of \see, to show $\Gamma,\phi\vdash\theta$ it suffices to show that $\Gamma,\phi(p_{i_k}\wedge\nem/[p_{i_k},m])\vdash\theta$ and $\Gamma,\phi(p_{i_k}\wedge\bot/[p_{i_k},m])\vdash\theta$. To show the latter, first note that by the assumption we have $\Gamma,\phi(\Theta^\ast_\emptyset/[p_{i_k},m])\vdash\theta$, i.e., $\Gamma,\phi(\bot/[p_{i_k},m])\vdash\theta$. It then suffices to check that $\phi(p_{i_k}\wedge\bot/[p_{i_k},m])\vdash\phi(\bot/[p_{i_k},m])$. But this follows  from \Cref{replacement_lm}, as by \ce we have $p_{i_k}\wedge\bot\vdash\bot$.

To show the former, in view of \sen it suffices to derive $\Gamma,\phi(p_{i_k}\wedge\Theta^\ast_X/[p_{i_k},m])\vdash\theta$ for all nonempty teams $X$ on $N\setminus\{i_k\}$. By the rule \dstrs$\wedge\sor$, we have $p_{i_k}\wedge\Theta^\ast_{X}\vdash\Theta^\ast_{Y}$, where $Y\subseteq 2^N$ is defined as
\[Y=\{s:N\to 2\mid s\upharpoonright N\setminus\{i_k\}\in X\text{ and }s(i_k)=1\}.\]
It then follows from \Cref{replacement_lm} that $\phi(p_{i_k}\wedge\Theta^\ast_X/[p_{i_k},m])\vdash\phi(\Theta^\ast_Y/[p_{i_k},m])$. On the other hand, clearly $Y\in \mathcal{Y}$ and the assumptions implies that $\Gamma,\phi(\Theta^\ast_Y/[p_{i_k},m])\vdash\theta$. Hence we obtain $\Gamma,\phi(p_{i_k}\wedge\Theta^\ast_X/[p_{i_k},m])\vdash\theta$, as desired.
\end{proof}

Now we are ready to prove the full statement for an arbitrary sequence $o$ of occurrences of atoms.

\begin{lemma}\label{general_se_multi_lm}
Let $\Lambda$ be the set of all strongly realizing sequences of $\phi$ over a sequence $o$ of some occurrences of atoms in a formula $\phi$. For any set $\Gamma\cup\{\theta\}$ of formulas in the language of \ECL, if $\Gamma,\phi^\ast_\Omega\vdash\theta$ for all $\Omega\in \Lambda$, then $\Gamma,\phi\vdash\theta$.
\end{lemma}
\begin{proof}
Let $o=\langle [\alpha_1,m_1],\,\dots,\,[\alpha_c,m_c]\rangle$. 
By the assumption, for any $Y_1$ that satisfies $\alpha_1$ we have
\[\begin{split}
\Gamma,\phi((\alpha_1)^\ast_{Y_1}/ [\alpha_1,m_1],(\alpha_2)^\ast_{X_2}/ [\alpha_2,m_2],\,&\dots,\,(\alpha_c)^\ast_{X_c}/[\alpha_c,m_c])\vdash\theta\\
&\text{ for all }\langle Y_1,X_2,\dots,X_c\rangle\in \Lambda
\end{split}\]
Then we conclude by \Cref{general_se_single_lm} that $\Gamma,\phi^\ast_\Omega\vdash\theta$ for all $\Omega\in \Lambda_1$, where $\Lambda_1$ is the set of all strongly realizing sequences of $\phi$ over a sequence $o_1=\langle [\alpha_2,m_2],\,\dots,\,[\alpha_c,m_c]\rangle$.

By repeating this argument $c$ times, we obtain $\Gamma,\phi\vdash\theta$ in the end.
\end{proof}

For the sake of the proof of the Completeness Theorem, we need to further transform each strong realization $\phi^\ast_\Omega$ into a formula $\Theta^\ast_{X_\Omega}$ in the normal form. To simplify notations, we write $\Theta^\ast_{\mathbf{0}}$ for the formula $\bot\wedge\nem$ and view $\mathbf{0}$ as a void team. Note that $\Theta^\ast_{\mathbf{0}}\neq\Theta^\ast_\emptyset=\bot$.

\begin{lemma}\label{WNF2SNF_ctl}
Let  $\Lambda$ be the set of all maximal strongly realizing sequences of a formula $\phi(p_{i_1},\dots,p_{i_n})$ in the language of \ECL.
\begin{description}
\item[(i)] For each $\Omega\in \Lambda$, we have \(\phi^\ast_\Omega\dashv\vdash\Theta^\ast_{X_\Omega}\) for some team $X_\Omega$  on $N=\{i_1,\dots,i_n\}$ or $X_\Omega=\mathbf{0}$.
\item[(ii)] Let $\Lambda_0=\{\Omega\in \Lambda\mid X_\Omega\neq\mathbf{0}\}$. We have \(\phi\,\equiv\,\bigbor_{\Omega\in \Lambda_0}\Theta^\ast_{X_\Omega}\).
\end{description}
\end{lemma}
\begin{proof}
(i) We prove the lemma by induction on the subformulas $\psi$ of $\phi$. If $\psi$ is an atom, then $\psi^\ast_\Omega=\Theta^\ast_{X_\Omega}$ for some team $X_\Omega$ on $N$ that satisfies $\psi$ and trivially $\psi^\ast_\Omega\dashv\vdash\Theta^\ast_{X_\Omega}$.

If $\psi=\delta\sor\chi$, then by the induction hypothesis, we have
\begin{equation}\label{SWNF_ctl_eq1}
\delta^\ast_{\Omega}\dashv\vdash\Theta^\ast_{X_{\Omega}}\text{ and }\chi^\ast_{\Omega}\dashv\vdash\Theta^\ast_{Y_{\Omega}},
\end{equation}
By \sorws we have  
\(\delta^\ast_{\Omega}\sor \chi^\ast_{\Omega}\dashv\vdash\Theta^\ast_{X_{\Omega}}\sor \Theta^\ast_{Y_{\Omega}}\). It then suffices to show that $\Theta^\ast_{X_{\Omega}}\sor \Theta^\ast_{Y_{\Omega}}\dashv\vdash\Theta^\ast_{Z}$ for some team $Z$ on $N$.

If $X_{\Omega}=\mathbf{0}$, then taking $Z=\mathbf{0}$, we derive $\Theta^\ast_{\mathbf{0}}\dashv\vdash\Theta^\ast_{\mathbf{0}}\sor \Theta^\ast_{Y_{\Omega}}$ by \sexfalso and  \sctrc. The case $Y_{\Omega}=\mathbf{0}$ is proved similarly. If $X_{\Omega}, Y_{\Omega}\neq\mathbf{0}$, then by \sorwk and \sorwe, we derive $ \Theta^\ast_{X_{\Omega}}\sor \Theta^\ast_{Y_{\Omega}}\dashv\vdash\Theta^\ast_{X_{\Omega}\cup Y_{\Omega}}$.

If $\psi=\delta\wedge\chi$, then  the induction hypothesis implies (\ref{SWNF_ctl_eq1}). By \ci and \ce we have 
\(\delta^\ast_{\Omega}\wedge \chi^\ast_{\Omega}\dashv\vdash\Theta^\ast_{X_{\Omega}}\wedge \Theta^\ast_{Y_{\Omega}}\). It then suffices to show that $\Theta^\ast_{X_{\Omega}}\wedge \Theta^\ast_{Y_{\Omega}}\dashv\vdash\Theta^\ast_{Z}$ for some team $Z$ on $N$.

If $X_{\Omega}=\mathbf{0}$, then taking $Z=\mathbf{0}$, we derive $\Theta^\ast_{\mathbf{0}}\dashv\vdash\Theta^\ast_{\mathbf{0}}\wedge \Theta^\ast_{Y_{\Omega}}$  by \sexfalso and  \ce. The case $Y_{\Omega}=\mathbf{0}$ is proved similarly. If $X_{\Omega}= Y_{\Omega}\neq\mathbf{0}$, then by \ce and \ci, we derive $ \Theta^\ast_{X_{\Omega}}\wedge \Theta^\ast_{Y_{\Omega}}\dashv\vdash\Theta^\ast_{X_{\Omega}}$.  If $X_{\Omega},Y_{\Omega}\neq\mathbf{0}$ and $X_{\Omega}\neq Y_{\Omega}$, then  we derive $\Theta^\ast_{\mathbf{0}}\dashv\vdash \Theta^\ast_{X_{\Omega}}\wedge \Theta^\ast_{Y_{\Omega}}$ by \sexfalso and \sctri.

(ii) It follows from the item  (i), the Soundness Theorem and  \Cref{WNF_ctl_semantic} that
\(\phi\,\equiv\, \bigbor_{\Omega\in \Lambda}\phi^\ast_\Omega\,\equiv\,\bigbor_{\Omega\in \Lambda}\Theta^\ast_{X_\Omega}.\)
If $\Lambda_0\neq\emptyset$, then the statement clearly follows, as $\Theta^\ast_{\mathbf{0}}\vee\psi=\big(\bot\wedge \nem\big)\bor\psi\equiv\psi$ for all formulas $\psi$. If $\Lambda_0=\emptyset$, then $\phi^\ast_\Omega\equiv\bot\wedge\nem$ for each $\Omega\in \Lambda$. Thus $\phi\equiv\bigbor_{\Omega\in \Lambda}\phi^\ast_\Omega\equiv\bigvee_{\Omega\in \Lambda}(\bot\wedge\nem)\equiv\bot\wedge\nem\equiv\bigvee\emptyset$.
\end{proof}


Finally, let us give the proof of the Completeness Theorem.

\begin{theorem}[Completeness Theorem]\label{completeness_ctl}
For any   formulas $\phi$ and $\psi$ in the language of \ECL, we have 
\(\phi\models\psi\,\Longrightarrow\,\phi\vdash\psi.\)
\end{theorem}
\begin{proof}
Suppose $\phi\models\psi$, where $\phi=\phi(p_{i_1},\dots,p_{i_n})$ and $\psi=\psi(p_{i_1},\dots,p_{i_n})$.
 By Lemma \ref{WNF2SNF_ctl} and the Soundness Theorem, we have
\begin{equation}\label{comp_ctl_eq1}
\phi\equiv\bigbor_{\Omega\in \Lambda_0}\Theta^{\ast}_{X_{\Omega}}\models\bigbor_{\Upsilon\in \Lambda_0'}\Theta^{\ast}_{Y_{\Upsilon}}\equiv\psi.
\end{equation}
where 
\begin{description}
\item[(i)] $\Lambda$, $\Lambda'$ are the sets of all strongly realizing sequences of $\phi$ and $\psi$, respectively and each $X_\Omega$ and $Y_\Upsilon$ are teams on $\{i_1,\dots,i_n\}$;
\item[(ii)] $\phi^\ast_\Omega\dashv\vdash\Theta^\ast_{X_\Omega}$ and $\psi^\ast_\Upsilon\dashv\vdash\Theta^\ast_{Y_\Upsilon}$ for all $\Omega\in\Lambda$ and $\Upsilon\in \Lambda'$;
\item[(iii)] $\Lambda_0=\{\Omega\in \Lambda\mid X_\Omega\neq\mathbf{0}\}$ and $\Lambda_0'=\{\Upsilon\in \Lambda'\mid Y_\Upsilon\neq\mathbf{0}\}$.
\end{description}

If $\Lambda_0,\Lambda_0'\neq\emptyset$, then for any $\Omega\in \Lambda$, by Lemma \ref{comp_CLTe_main_lm} there exists $\Upsilon\in \Lambda'$ such that $X_\Omega=Y_\Upsilon$. We then have
\[\phi^\ast_\Omega\vdash\Theta^{\ast}_{X_{\Omega}}=\Theta^{\ast}_{X_{\Upsilon}}\vdash \psi^\ast_\Upsilon\vdash\psi\]
by (ii) and \Cref{se_2_form}. Finally,  we obtain $\phi\vdash\psi$ by \Cref{general_se_multi_lm}.

If $\Lambda_0=\emptyset$, then for each $\Omega\in\Lambda$ we have $\phi^\ast_\Omega\dashv\vdash\Theta^\ast_{\mathbf{0}}=\bot\wedge\nem$. Then by \sexfalso we derive $\phi^\ast_\Omega\vdash\psi$ and $\phi\vdash\psi$ follows from \Cref{general_se_multi_lm} again. If $\Lambda'_0=\emptyset$, then  $\psi\equiv\bigvee\emptyset=\bot\wedge\nem$. But in view of (\ref{comp_ctl_eq1}), we must also have $\phi\equiv\bot\wedge\nem$ and $\Lambda_0=\emptyset$. This then reduces to the previous case.
\end{proof}

\subsection{\PInem, \PIncs and other extensions of \ECL}

The argument in the previous section can also be applied to axiomatize other propositional team logics obtained by adding new atoms with the empty team property to the language of \ECL, such as strong propositional independence logic (\PInem) and strong propositional inclusion logic (\PIncs). Throughout the section, we write \ECLp for an arbitrary such logic.   In this section, we will show how to generalize the method in the previous section to axiomatize \ECLp, and \PInem and \PInc in particular. 





What was crucial in the axiomatization of \ECL was the notion of a \emph{strong realization} of a formula. This notion can be generalized to richer languages, such as \PInem and \PIncs. A \emph{strong realization} $\alpha^\ast_Y$ of an atom $\alpha$ (such as independence atom and inclusion atom) over a team $Y$ that satisfies $\alpha$ is defined as $\alpha^\ast_Y:=\Theta^\ast_Y$.  A \emph{strongly realizing sequence of a formula $\phi$ over a sequence $o$} of some occurrences of atoms in $\phi$ and a \emph{strong realization} (over $o$) are defined the same way as in the logic \ECL, except that a richer language \LL may contain more atoms.

Our natural deduction system of \LL consists of all of the axioms and the rules from the system of \ECL (\Cref{Natrual_Deduct_ctl}), together with the Introduction Rule and the Elimination Rule (\alphai and \sealpha  presented below) that characterize the equivalence between an arbitrary new atom $\alpha$ and the Boolean disjunction $\bigvee_Y\Theta^\ast_Y$ of its strong realizations. 

\begin{definition}[A natural deduction system of \ECLp]\label{Natrual_Deduct_PInem}\

\begin{center}
\setlength{\tabcolsep}{6pt}
\renewcommand{\arraystretch}{1.8}
\setlength{\extrarowheight}{1pt}
\begin{tabular}{|C{0.95\linewidth}|}
\multicolumn{1}{c}{\textbf{AXIOM}}\\\hline\hline
Atomic excluded middle\\
\AxiomC{}\noLine\UnaryInfC{} \RightLabel{\exclmid}\UnaryInfC{$p_i\sor\neg p_i$}\noLine\UnaryInfC{}\DisplayProof\\\hline
\end{tabular}
\end{center}


\begin{center}
\setlength{\tabcolsep}{6pt}
\renewcommand{\arraystretch}{1.8}
\setlength{\extrarowheight}{2pt}
\begin{tabular}{|C{0.95\linewidth}|}
 \multicolumn{1}{c}{\textbf{RULES}}\\\hline\hline
All of the rules from \Cref{Natrual_Deduct_ctl}, together with the following two rules for each new atom $\alpha$ in \ECLp:\\\hline
Atom $\alpha$ introduction\\
\AxiomC{}\noLine\UnaryInfC{$D$}\noLine\UnaryInfC{$\Theta^\ast_Y$} \RightLabel{ \alphai}\UnaryInfC{$\alpha$} 
\DisplayProof \\
{\small where  $Y$ is a team on a set $N$ of indices that satisfies $\alpha$}\footnotemark\\\hline
Strong elimination rule for $\alpha$\\
\def\ScoreOverhang{0.5pt}
\def\defaultHypSeparation{\hskip .1in}
\AxiomC{$D_0$}\noLine\UnaryInfC{$\phi$} \AxiomC{}\noLine\UnaryInfC{[$\phi(\Theta^\ast_{Y_1}/[\alpha,m])$]}\noLine\UnaryInfC{$D_1$} \noLine\UnaryInfC{$\theta$} \AxiomC{\null}\noLine\UnaryInfC{\null}\noLine\UnaryInfC{$\ldots$}\AxiomC{[$\phi(\Theta^\ast_{Y_k}/[\alpha,m])$]}\noLine\UnaryInfC{$D_k$}\noLine\UnaryInfC{$\theta$} \RightLabel{\sealpha}\QuaternaryInfC{$\theta$} \DisplayProof\\
{\small where $\{Y_1,\dots,Y_k\}$ is the set of all teams on a set $N$ of indices that satisfy $\alpha$}
\\\hline
\end{tabular}
\end{center}
\end{definition}

\footnotetext{On the surface this looks like a confusion between syntax and semantics. However, it is as in the Conjunction Introduction Rule ``From $\phi$  and $\psi$ we infer $\phi\wedge\psi$". On the basic level the rules establish a connection between logical operations and their intended meaning. We could replace here the assumption ``$Y$ satisfies the atom $\alpha$" by explicitly listing teams on $N$ that satisfy $\alpha$, but that would be more cumbersome.}

Next, we prove the Soundness Theorem and the (Strong) Completeness Theorem for the system.

\begin{theorem}[Soundness Theorem]\label{pi_soundness}
For any  set $\Gamma\cup\{\phi\}$ of formulas in the language of \ECLp, we have 
\(\Gamma\vdash\phi\,\Longrightarrow\,\Gamma\models\phi.\)
\end{theorem}
\begin{proof}
We show that for each derivation $D$ with the conclusion $\phi$ and the hypotheses in $\Gamma$ we have  $\Gamma\models\phi$.
We only verify the cases when the rules \alphai and \sealpha are applied. 

\alphai: Assume that $D$ is a derivation for $\Pi\vdash\Theta^\ast_Y$ where $Y$ is a team on $N$ that satisfies the atom $\alpha$. We show that $\Pi\models \alpha$ follows from the induction hypothesis $\Pi\models\Theta^\ast_Y$. This is reduced to showing $\Theta^\ast_Y\models\alpha$. For any team $X$ such that $X\models \Theta^\ast_Y$, by \Cref{comp_CLTe_main_lm} we have $X=Y$. Thus $X\models\alpha$ follows from the assumption.

\sealpha: Put $\phi^\ast_i=\phi(\Theta^\ast_{Y_i}/[\alpha,m])$ for each $i\in\{1,\dots,k\}$. Assume that $D_0,D_1,\dots,D_k$ are derivations for $\Pi_0\vdash\phi$,\, $\Pi_1,\phi^\ast_1\vdash\theta$, ..., $\Pi_k,\phi^\ast_k\vdash\theta$, respectively. We show that $\Pi_0,\Pi_1,\dots,\Pi_k\models\theta$ follows from the induction hypothesis $\Pi_0\models\phi$,\, $\Pi_1,\phi^\ast_1\models\theta$, ..., $\Pi_k,\phi^\ast_k\models\theta$. This is reduced to showing that $\phi\models\phi^\ast_1\vee\dots\vee\phi^\ast_k$ by induction on $\phi$. If $\phi$ is the atom $\alpha$, then the statement follows from  \Cref{comp_CLTe_main_lm} and the choice of the $Y_i$'s. The other cases are left to the reader.
\end{proof}


\begin{theorem}[Completeness Theorem]\label{completeness_pi}
For any  formulas $\phi$ and $\psi$  in the language of \ECLp, we have 
\(\phi\models\psi\,\Longrightarrow\,\phi\vdash\psi.\)
\end{theorem}
\begin{proof}
The theorem is proved by a similar argument to that of the proof of \Cref{completeness_ctl}. Especially, all the lemmas leading to \Cref{completeness_ctl} can be easily generalized to the case of any extension \ECLp of \ECL with new atoms having the empty team property. In particular, in the proof of the lemma that corresponds to \Cref{se_2_form}, if $\psi$ is a new atom $\alpha$, then $\psi^\ast_\Omega\vdash\psi$ follows from \alphai. When proving the lemma that corresponds to \Cref{general_se_single_lm}, one applies \sealpha when $\alpha$ is a new atom.
\end{proof}



Instantiating the new atoms $\alpha$ in the deduction system of \Cref{Natrual_Deduct_PInem} with independence atoms or inclusion atoms, we obtain sound and (strongly) complete deduction systems of \PInem and \PIncs. We end this section with a demonstration of the natural deduction system of \PInem. Below we derive the Geiger-Paz-Pearl axioms  \cite{Geiger-Paz-Pearl_1991}. 
To simplify notation, we write \indi for the Independence Atom Introduction Rule and \sei for the Independence Atom Elimination Rule.

\begin{example}\label{Geiger-Paz-Pearl_axiom_pinemz}
Let $\vec{x}=p_{i_1}\cdots p_{i_k}$, $\vec{y}=p_{j_1}\cdots p_{j_m} $ and $\vec{z}=p_{l_1}\cdots p_{l_n}$. The following Geiger-Paz-Pearl axioms are derivable in the natural deduction system of \PInem:
\begin{description}
\item[(i)] $\vec{x}\perp\vec{y}~\vdash\vec{y}\perp\vec{x}$
\item[(ii)] $\vec{x}\perp\vec{y}~\vdash\vec{z}\perp\vec{y}$, where $\vec{z}$ is a subsequence of $\vec{x}$.
\item[(iii)] $\vec{x}\perp\vec{y}~\vdash\vec{u}\perp\vec{v}$, where $\vec{u}$ is a permutation of $\vec{x}$ and $\vec{v}$ is a permutation of $\vec{y}$.
\item[(iv)] $\vec{x}\perp \vec{y},~\vec{x}\vec{y}\perp\vec{z}~\vdash\vec{x}\perp \vec{y}\vec{z}$.
\end{description}
\end{example}
\begin{proof} We are going to use the Independence Atom Introduction Rule \indi\ a lot here. Therefore the proofs below seem entirely semantical. However, we have built the meaning of the independence atom into the rule \indi, so it is only natural that we refer to this meaning in the proofs. This is more complicated but, in principle, analogous to the way we often use in elementary logic the meaning of ``and" and ``or" when we prove e.g. distributivity laws using just the Elimination Rule and Introduction Rule for $\wedge$ and $\vee$.

Put $K=\{i_1,\dots,i_k\}$, $M=\{j_1,\dots,j_m\}$ and $N=\{l_1,\dots,l_n\}$.  

(i) By \sei, it suffices to show that for any team $Y$ on $K\cup M$ such that $Y\models\vec{x}\perp\vec{y}$ we have $\Theta^\ast_Y\vdash\vec{y}\perp\vec{x}$. But this follows from \indi, as $Y\models \vec{y}\perp\vec{x}$ also holds.

(ii) By \sei, it suffices to show that for any team $Y$ on $K\cup M\cup N$ such that $Y\models\vec{x}\perp\vec{y}$ we have $\Theta^\ast_Y\vdash\vec{z}\perp\vec{y}$, where $\vec{z}$ is a subsequence of $\vec{x}$. In view of \indi, this is reduced to showing $Y\models \vec{z}\perp\vec{y}$. But this is obvious.

(iii) By \sei, it suffices to show that for any team $Y$ on $K\cup M$ such that $Y\models\vec{x}\perp\vec{y}$ we have $\Theta^\ast_Y\vdash \vec{u}\perp\vec{v}$, where $\vec{u}=p_{i_1}\cdots p_{i_a}$ is a permutation of $\vec{x}$ and $\vec{v}=p_{m_1}\cdots p_{m_b}$ is a permutation of $\vec{y}$. In view of \indi this is reduced to showing $Y\models \vec{u}\perp\vec{v}$, which follows from that $Y\models\vec{x}\perp\vec{y}$.


(iv) We will show that $(\vec{x}\perp \vec{y})\wedge(\vec{x}\vec{y}\perp\vec{z})\vdash\vec{x}\perp \vec{y}\vec{z}$. By \sei, it suffices to show that for any team $X$ on $K\cup M\cup N$ such that $X\models\vec{x}\perp\vec{y}$, we have  $\Theta^\ast_X\wedge(\vec{x}\vec{y}\perp\vec{z})\vdash\vec{x}\perp\vec{y}\vec{z}$. But this, by \sei, is further reduced to showing that for any team $Y$ on $K\cup M\cup N$ such that  $Y\models\vec{x}\vec{y}\perp\vec{z}$, we have $\Theta^\ast_X\wedge\Theta^\ast_Y\vdash\vec{x}\perp\vec{y}\vec{z}$.

Now, if $X\neq Y$, then by \sctri and \sexfalso, we derive  $\Theta^\ast_X\wedge\Theta^\ast_Y\vdash\bot\wedge\nem\vdash\vec{x}\perp\vec{y}\vec{z}$.

If $X=Y$, then $\Theta^\ast_X=\Theta^\ast_Y$. By  \indi, it suffices to show $X\models\vec{x}\perp\vec{y}\vec{z}$. For any $s_1,s_2\in X$, since $X\models\vec{x}\perp\vec{y}$, there exists $s_3\in X$ such that $s_1(\vec{x})=s_3(\vec{x})$ and $s_2(\vec{y})=s_3(\vec{y})$. But as $X\models\vec{x}\vec{y}\perp\vec{z}$, there exists $s_4\in X$ such that $s_4(\vec{x})s_4(\vec{y})=s_3(\vec{x})s_3(\vec{y})$ and $s_4(\vec{z})=s_2(\vec{z})$. We then conclude that $s_4(\vec{x})=s_3(\vec{x})=s_1(\vec{x})$ and $s_4(\vec{y})s_4(\vec{z})=s_3(\vec{y})s_2(\vec{z})=s_2(\vec{y})s_2(\vec{z})$. Hence $X\models\vec{x}\perp\vec{y}\vec{z}$ and this completes the proof.
\end{proof}

\section{Concluding remarks}\label{conc}

In our previous work \cite{VY_PD} we have investigated classical propositional logic, and several versions of propositional dependence logic. We proved their expressive completeness for downward closed team properties, and gave several complete axiomatizations of such logics. In this paper, we have studied propositional team logics more generally, recognizing that there is a whole hierarchy of them. We have established the results in \Cref{fig:exp_pw} concerning the expressive power of these logics. Several expressively complete logics are identified. For example, we have proved that \PU is  expressively complete for the set of union closed team properties. We also derived normal forms for many of the logics. Some of the logics we considered have the empty team property, and some do not.  We axiomatized the logics (\PT, \ECL, \PInem and \PIncs) without the empty team property, and we leave the concrete axiomatization of \PU  for future research.
As is reflected by the sophisticated Strong Elimination Rules we gave in this paper, propositional logics of independence are more intricate than propositional logics of dependence.  In particular, we feel that we have not yet fully understood the notion of independence, and neither a characterization of expressive power nor a complete axiomatization is given for \PI here. Nevertheless, it is our hope that the results obtained concerning logics around \PI will set the stage for further research in this field and lead to a better understanding of \PI, propositional independence logic, itself.

\section*{Acknowledgements}
The authors would like to thank Maria Aloni, Pietro Galliani, Lauri Hella, Juha Kontinen, V\'{i}t Pun\v{c}och\'{a}\v{r}, Floris Roelfsen and Dag Westerst\aa hl for useful conversations related to this paper.




\bibliographystyle{acm}


%
%
%

\clearpage

\end{document}